\documentclass{article}
\usepackage{amsmath,amssymb,bbm,chicago,graphics,graphicx,latexsym,subfigure}
\title{Smooth plug-in inverse estimators in the current status continuous mark model.}
\author{P. Groeneboom$^{\rm a}$,\ G. Jongbloed$^{\rm a}$ \& B.I. Witte$^{\rm b}$ \\ {$^{\rm a}$\it Delft University of Technology; $^{\rm b}$VU University Medical Center}}

\def\A{\ensuremath{\mathcal{A}}}
\def\B{\ensuremath{\mathcal{B}}}
\def\E{\ensuremath{{\mbox{E}}}}
\def\M{\ensuremath{\mathcal M}}
\def\N{\ensuremath{\mathcal N}}
\def\cP{\ensuremath{\mathcal P}}
\def\R{\ensuremath{{\mbox{I\!R}}}}
\def\cS{\ensuremath{\mathcal S}}
\def\U{\ensuremath{\mathcal U}}

\def\conv{\ensuremath{\mathop{\rm \longrightarrow}}}

\newcommand{\an}{\ensuremath{\alpha_n}}
\newcommand{\bn}{\ensuremath{\beta_n}}
\newcommand{\Fpi}[1]{\ensuremath{\hat F_n^{(#1)}}}
\newcommand{\fpi}[1]{\ensuremath{\hat f_n^{(#1)}}}
\newcommand{\FMS}{\ensuremath{\hat F_n^{MS}}}
\newcommand{\gh}[1]{\ensuremath{\hat g_{#1}}}
\newcommand{\hf}{\ensuremath{h_{f_0}}}
\newcommand{\Hn}[1]{\ensuremath{{\rm I\!H}_{#1}}}

\newcommand{\cov}{\ensuremath{{\mbox{Cov}}}}
\newcommand{\Dconv}{\ensuremath{\leadsto}}
\newcommand{\norm}[2]{\mbox{$\|#1\|_{#2}$}}
\newcommand{\Pconv}{\ensuremath{\stackrel{\cP}{\conv}}}
\newcommand{\Var}{\ensuremath{{\mbox{Var}}}}
\newcommand{\eps}{\ensuremath{\varepsilon}}

\newcommand{\be}{\begin{eqnarray*}}
\newcommand{\ee}{\end{eqnarray*}}
\newcommand{\bef}{\begin{eqnarray}}
\newcommand{\eef}{\end{eqnarray}}

\renewcommand{\O}{\ensuremath{{\cal O}}}
\renewcommand{\o}{{\mbox{\scriptsize $\O$}}}

\newtheorem{theorem}{Theorem}
\newtheorem{lemma}[theorem]{Lemma}
\newtheorem{corollary}[theorem]{Corollary}

\newenvironment{proof}{\noindent\textit{Proof:}}{$\hfill\square$\vspace{3mm}}
\newenvironment{proof2}[1]{\noindent\textit{Proof of #1:}}{$\hfill\square$\vspace{3mm}}
\newenvironment{remark}{\noindent\textit{Remark.}}{\vspace{3mm}}

\begin{document}
\maketitle

\noindent ABSTRACT. We consider the problem of estimating the joint distribution function of the event time and a continuous mark variable when the event time is subject to interval censoring case 1 and the continuous mark variable is only observed in case the event occurred before the time of inspection. The nonparametric maximum likelihood estimator in this model is known to be inconsistent. We study two alternative smooth estimators, based on the explicit (inverse) expression of the distribution function of interest in terms of the density of the observable vector. We derive the pointwise asymptotic distribution of both estimators.

\vspace{5mm}

\section{Introduction}
To test the efficacy of a vaccine, preventative trials are held where participants are injected with the vaccine and tested for several times. One of the questions of interest in the trials is whether the efficacy depends on the genetic sequence of the exposing virus. To answer this question, \shortciteN{rgp120:05} studied the so-called viral distance between the HIV sequence represented in the vaccine and the HIV sequence the participant is infected with. This distance can be considered as a ``mark'' variable, since it can only be observed if infection has already taken place. This variable is possibly correlated with the time of HIV infection and according to \shortciteN{gilbert:01} it is natural to treat it as a continuous random variable.

A natural statistical model to describe the observations in these HIV vaccine trials is the interval censored continuous mark model, which was first studied by \shortciteN{hudgens:07}. In this model, $X$ is an event time (the time of HIV infection) and $Y$ is a continuous mark variable (the viral distance) and we are interested in the bivariate distribution function $F_0$ of the pair $(X,Y)$. However, the event time is subject to interval censoring case $k$. We restrict ourselves to the special instance of interval censoring case 1 (also known as current status censoring) and refer to this model as the current status continuous mark model.

For this model, the method of nonparametric maximum likelihood estimation is studied \citeN{maathuis:08}. There it is proved that the maximum likelihood estimator (MLE) is inconsistent. An approach they propose to `repair' the inconsistency is by discretizing the mark variable. Discretizing the mark variable to $K$ levels, the resulting observations can be viewed as observations from the current status $K$-competing risk model. The characterization, consistency and (local) asymptotic distribution theory of the MLE in that model follow from \shortciteANP{groeneboom:08} \citeyear{groeneboom:08,groeneboom:08b}. Results on consistency and asymptotics as $K\to\infty$ are not yet known.

Another natural way to estimate the distribution function $F_0$ is by viewing this problem as an inverse statistical model. In inverse models, like interval censoring models or deconvolution models, one is interested in estimating the distribution of a random variable $X$. Instead of observing this variable $X$ directly, only a related variable $W$ is observed. The distribution of $W$ depends on the distribution function $F_0$ of $X$ (or its Lebesgue density $f_0$) via a known (direct) relation. In some cases, this relation can be explicitly inverted to express $F_0$ in terms of the distribution of $W$, and to estimate $F_0$ one can plug in an estimator for the distribution of $W$ in this inverse relation. The resulting estimator is called a plug-in inverse estimator. Plug-in inverse estimators are studied by \citeN{hall:88} in Wicksell's corpuscle problem, by \citeN{stefanski:90} in the deconvolution model and by \citeN{burke:88} in the bivariate right-censoring model.

In this paper we study plug-in inverse estimators in the current status continuous mark model. We start with a formal description of the model and define two plug-in inverse estimators in Section \ref{sec:def_est}. One estimator is based on univariate kernel smoothing, the other is based on bivariate kernel smoothing. In Section \ref{sec:cons_Fpi2}, we prove that these estimators are uniformly consistent for $F_0$. Unfortunately, these estimators are not monotonically increasing in both directions, which is a necessary property of bivariate distribution functions. In Section \ref{sec:cons_Fpi2} we prove that the estimator based on bivariate kernel smoothing asymptotically will have all properties of a bivariate distribution function on a large subset of $[0,\infty)^2$. The plug-in inverse estimator resulting from the univariate kernel smoothing estimator is computationally and asymptotically more tractable. In Section \ref{sec:asymp_Fpi1_Fpi2}, we first derive the asymptotic distribution of this estimator. After that, we prove that for certain choices of the smoothing parameter in the $z$-direction, the two plug-in inverse estimators are asymptotically equivalent, while for other choices the asymptotic biases differ but the asymptotic variances are equal. This phenomenon was also observed by \citeN{marron:87} and \shortciteN{patil:94} in the case of estimating densities based on right-censored data and by \shortciteN{witte:10} in the current status model. The asymptotic distribution of the estimator based on bivariate kernel smoothing then follows easily. In Section \ref{sec:smooth_functionals}, we briefly address the problem of estimating smooth functionals. A small simulation study to compare the estimators with the binned MLE studied by \citeN{maathuis:08} and the maximum smoothed likelihood estimator studied by \shortciteN{witte:10} is performed in Section \ref{sec:sim_study}. Technical proofs and lemmas can be found in the Appendix.

\section{Definition of the estimators}\label{sec:def_est}
In this section we describe the current status continuous mark model in more detail and define two plug-in inverse estimators based on kernel smoothing.

Let $X$ be an event time, $Y$ a continuous mark variable and $F_0$ be the distribution function of the pair $(X,Y)$. In the current status continuous mark model, instead of observing the pair $(X,Y)$, we observe a censoring variable $T$, independent of $(X,Y)$ with Lebesgue density $g$, as well as the indicator variable $\Delta=1_{\{X\leq T\}}$. In case $X\leq T$, i.e.\ if $\Delta=1$, we also observe the mark variable $Y$; in case $\Delta=0$ the variable $Y$ is not observed. Under the assumption that $P(Y=0)=0$, we can represent the observable information on $(X,Y)$ in the vector $W=(T,Z,\Delta)$, for $Z = \Delta \cdot Y$.

Let $\lambda_i$ be Lebesgue-measure on \R$^i$, $\B$ the Borel $\sigma$-algebra on $[0,\infty)^2$ and define the measure $\lambda$ on \B\ by
\be\lambda\big(B\big) = \lambda_2\big(B\big) + \lambda_1\big(\{x\in[0,\infty):(x,0)\in B\}\big), \ B\in\B.\ee
Then, the density of the observable vector $W$ w.r.t.\ the product of this measure with counting measure on $\{0,1\}$ can be written as
\bef\label{def:obs_dens}
h_{F_0}(t,z,\delta) = \delta g(t)\partial_2 F_0(t,z) + (1-\delta)g(t)\big(1-F_{0,X}(t)\big) = \delta h_1(t,z) + (1-\delta)h_0(t),
\eef
where $F_{0,X}$ is the marginal distribution of $X$ and $\partial_2 F_0(t,z) = \frac{\partial}{\partial z} F_0(t,z)$. More generally, for convenience of notation, we denote the $j$th partial derivative with respect to $x_i$ of a function $F$ by $\partial_i^j F$, i.e.
\be \partial_i^j F(x_1,x_2) = \frac{\partial^j}{\partial y_i^j} F(y_1,y_2)\Big|_{(y_1,y_2)=(x_1,x_2)},\ee
and omit $j$ when $j=1$.

Based on the relation $h_1(t,z) = g(t)\partial_2 F_0(t,z)$, we can express the bivariate distribution function $F_0$ of $(X,Y)$ in terms of the (sub-)densities $g$ and $h_1$
\bef\label{def:inv_rel_cscm}
F_0(t,z) = \frac{1}{g(t)}\int_0^zh_1(t,v)\,dv.
\eef
Then, our plug-in inverse estimator in the current status continuous mark model is defined as
\be \hat F(t,z) = \frac{1}{\hat g(t)}\int_0^z\hat h_1(t,v)\,dv,\ee
where $\hat g$ and $\hat h_1$ are estimators for $g$ and $h_1$, respectively.

Before explicitly choosing the estimators $\hat g$ and $\hat h_1$, we introduce some notation. Throughout the paper $k$ denotes a univariate kernel density, $\tilde k$ a bivariate kernel density and $(\an)$ and $(\bn)$ vanishing sequences of positive smoothing parameters. Let $k_{\an}$ and $\tilde k_{\an,\bn}$ the rescaled versions of $k$ and $\tilde k$, i.e., $k_{\an}(u) = \an^{-1}k(u/\an)$ and $\tilde k_{\an,\bn}(u,v) = \an^{-1}\bn^{-1}\tilde k(u/\an,v/\bn)$. Furthermore, we define
\be m_2(k) = \int u^2 k(u)\,du, \ m_2(\tilde k) = \iint w_1^2 \tilde k(w_1,w_2)\,dw_1\,dw_2.\ee

Then for fixed $t_0$ and $z_0$, we estimate $g$ and $h_1$ by their respective univariate and bivariate kernel (sub-)density estimators
\be\gh{n}(t_0) = \frac1n\sum_{i=1}^n k_{\an} (t_0-T_i), \ \hat h_{n,1}^{(2)}(t_0,z_0) = \frac1n\sum_{i=1}^n \Delta_i\tilde k_{\an,\bn} (t_0-T_i,z_0-Z_i).\ee
The plug-in inverse estimator then becomes
\bef\label{def:Fn_two}
\Fpi{2}(t_0,z_0) = \frac{\int_0^{z_0}\frac1n \sum_{i=1}^n \Delta_i\tilde k_{\an,\bn} (t_0-T_i,z-Z_i)\,dz}{\frac1n\sum_{i=1}^n k_{\an}(t_0-T_i)}.
\eef
Here, superscript 2 in the notation for the plug-in estimator refers to the fact that there is smoothing in two directions.

In Section \ref{sec:asymp_Fpi1_Fpi2} we also consider a less natural, but computationally and asymptotically more tractable estimator using an estimate for the numerator $\int_0^{z_0}h_1(t_0,z)\,dz$ based on smoothing only in the $t$-direction, i.e., when we estimate it by
\be \frac1n \sum_{i=1}^n 1_{[0,z_0]}(Z_i)\Delta_i k_{\an} (t_0-T_i).\ee
The plug-in inverse estimator then becomes
\bef\label{def:Fn_one}
\Fpi{1}(t_0,z_0) = \frac{\frac1n \sum_{i=1}^n 1_{[0,z_0]}(Z_i) \Delta_i k_{\an} (t_0-T_i)}{\frac1n\sum_{i=1}^n k_{\an}(t_0-T_i)}.
\eef
Superscript 1 in the notation for this estimator refers to the fact that there is only smoothing in one direction. Note that if we take $k(y) = \frac12 1_{[-1,1]}(y)$, (\ref{def:Fn_one}) results in
\be \Fpi{1}(t_0,z_0) = \frac{\int_{u\in A_n}\int_{z\leq z_0}\delta\,d\Hn{n}(u,z,\delta)}{\int_{u\in A_n}\int_{z\geq0}\,d\Hn{n}(u,z,\delta)},\ee
where \Hn{n}\ is the empirical distribution of the observations $(T_1,Z_1,\Delta_1),\ldots,(T_n,Z_n,\Delta_n)$ and $A_n=A_n(t_0)=[t_0-\an,t_0+\an]$. This estimator is the total number of observations $T_i$ in $A_n$ with $Z$-value smaller than or equal to $z_0$ and $\Delta=1$ divided by the total number of observations $(T_i,Z_i)$ in the strip $A_n\times [0,\infty)$.

It is very natural to define the kernel density $k$ in terms of the kernel density $\tilde k$ as stated in assumption $(K.1)$:
\begin{itemize}
\item[$(K.1)$] Let $\tilde k$ be a bivariate kernel density, then the kernel density $k$ is defined as
\be k(w_1)=\int \tilde k(w_1,w_2)\,dw_2.\ee
\end{itemize}
Indeed, if $(K.1)$ holds the estimator \Fpi{2}\ also satisfies the inverse relation $h_0(t) = g(t)\big(1-F_{0,X}(t)\big)$ that follows from substituting $\delta=0$ in (\ref{def:obs_dens}). To see this, note that we have that
\be \hat g_n(t_0) &=& \frac1n\sum_{i=1}^n k_{\an}(t_0-T_i) = \frac1n\sum_{i=1}^n (1-\Delta_i) k_{\an}(t_0-T_i) + \frac1n\sum_{i=1}^n \Delta_i k_{\an}(t_0-T_i)\\
&=& \frac1n\sum_{i=1}^n (1-\Delta_i) k_{\an}(t_0-T_i) + \int\frac1n\sum_{i=1}^n \Delta_i \tilde k_{\an,\bn}(t_0-T_i,z-Z_i)\,dz. \ee
If we define $\hat h_{n,0}(t_0) = \frac1n\sum_{i=1}^n (1-\Delta_i) k_{\an}(t_0-T_i)$ as an estimator for the sub-density $h_0$ in (\ref{def:obs_dens}), then 
\be 1-\hat F_{n,X}^{(2)}(t_0) = 1-\Fpi{2}(t_0,\infty) = 1- \frac{\int_0^{\infty}\hat h_{n,1}(t_0,z)\,dz}{\hat g_n(t_0)} = 1-\frac{\hat g_n(t_0) - \hat h_{n,0}(t_0)}{\hat g_n(t_0)} = \frac{\hat h_{n,0}(t_0)}{\hat g_n(t_0)}.\ee

Figure \ref{fig:ill_Fn1} illustrates the estimator \Fpi{1} for $n=10$ and $n=100$. For $F_0$ we took the uniform distribution on $[0,1]^2$ and for $g$ the uniform distribution on $[0,1]$. As kernel density we used $k(y) = \frac121_{[-1,1]}(y)$. The smoothing parameter \an\ is taken to be $0.65$ for $n=10$ and $0.40$ for $n=100$. These values are chosen for illustrative purpose only and do not depend on the data. In Section \ref{sec:sboot} we briefly address the problem of choosing \an\ and \bn\ depending on the data.

\begin{center}
{\bf [Figure \ref{fig:ill_Fn1} here]}
\end{center}
Note that these estimators are not true bivariate distribution functions, as they decrease locally in the $x$-direction. Monotonicity of a bivariate function is a necessary (but not sufficient) condition in order to be a bivariate distribution function, hence these estimators can be seen as naive estimators. The estimator \Fpi{2}\ can also have this undesirable naive behavior.

\section{Consistency and monotonicity}\label{sec:cons_Fpi2}
In this section we prove that the estimators \Fpi{1} and \Fpi{2} are uniformly consistent. Furthermore, we prove that for appropriate choices of the bandwidths and $n$ sufficiently large, \Fpi{2}\ will have all properties of a bivariate distribution function on a large subset of $[0,\infty)^2$, with arbitrarily high probability. To derive these results for \Fpi{2}, we assume the distribution function of interest $F_0$ and the censoring density $g$ satisfy the following conditions.
\begin{itemize}
\item[$(F.1)$] The Lebesgue density $f_0$ of $F_0$ exists for all $(t,z)\in[0,\infty)^2$.
\item[$(G.1)$] Let $\cS_{0,X}^{\circ}$ denote the interior of the support of the marginal density $f_{0,X}$ of $X$. On $\cS_{0,X}^{\circ}$, the density $g$ satisfies $0< g<\infty$ and its derivative $g'$ is uniformly continuous and bounded.
\end{itemize}

We also impose some conditions on the kernel densities $k$ and $\tilde k$, as well as a condition on the smoothing parameters \an\ and \bn.
\begin{itemize}
\item[$(K.2)$] The kernel density $k$ has compact support $[-1,1]$, is continuous and symmetric around 0.
\item[$(K.3)$] The kernel density $\tilde k$ has compact support $[-1,1]^2$, is continuous and satisfies
\be \iint w_i\tilde k(w_1,w_2)\,dw_1\,dw_2 = 0 \ (i=1,2), \ \iint w_2^2\tilde k(w_1,w_2)\,dw_1\,dw_2 = \iint w_1^2\tilde k(w_1,w_2)\,dw_1\,dw_2.\ee
\item[$(C.1)$] The positive smoothing parameters \an\ and \bn\ satisfy
\be \lim_{n\to\infty} \an = \lim_{n\to\infty} \bn = 0,\ \lim_{n\to\infty} n\an = \infty.\ee
\end{itemize}
A possible choice for the bivariate kernel density $\tilde k$ is the product kernel density $\tilde k(x,y) = k_1(x)k_2(y)$ for univariate kernel densities $k_1$ and $k_2$ with compact support $[-1,1]$ that are continuous and symmetric around 0. This kernel density $\tilde k$ satisfies condition $(K.1)$ for $k=k_1$ and $(K.3)$ if $m_2(k_1)=m_2(k_2)$.

\begin{theorem}
Assume $F_0$ and $g$ satisfy conditions $(F.1)$ and $(G.1)$. Also assume $k$ is defined via relation $(K.1)$ and satisfies condition $(K.2)$. Furthermore, let \an\ and \bn\ satisfy condition $(C.1)$. Let $\A\subset\R_+^2$ be a compact set such that $g(t)\geq c>0$ for all $(t,z)\in\A$
Then \Fpi{1}\ and \Fpi{2} are uniformly consistent on $\A$.
\end{theorem}

\begin{proof}
The uniform consistency of \Fpi{1}\ follows from Theorem 3.2 in \shortciteN{haerdle:88}.

To prove that \Fpi{2}\ is uniformly consistent on \A, first note that for $n$ sufficiently large there exists $\eps>0$ such that
\be\sup_{(t,z)\in\A} \big|\hat h_{n,1}^{(2)}(t,z) - h_1(t,z)\big| \leq \eps,\ee
see also Lemma \ref{lem:unif_cons_hn2_gn}. Hence
\be \left|\int_0^z\hat h_{n,1}^{(2)}(t,y)\,dy - \int_0^zh_1(t,y)\,dy\right| \leq \int_0^z\big|\hat h_{n,1}^{(2)}(t,y)-h_1(t,y)\big|\,dy \leq \eps z.\ee
Since $z\in\A$ and $\A$ is compact, this implies that
\bef\label{unif_conv:num_Fpi2}
\sup_{(t,z)\in\A}\left|\int_0^z\hat h_{n,1}^{(2)}(t,y)\,dy - \int_0^zh_1(t,y)\,dy\right| \Pconv 0.
\eef

Write $N_n^{(2)}(t,z) = \int_0^z\hat h_{n,1}^{(2)}(t,y)\,dy$ and $N(t,z)=\int_0^z h_1(t,y)\,dy$. Then we have that
\be \big|\Fpi{2}(t,z)-F_0(t,z)\big| &=& \left|\frac{N_n^{(2)}(t,z)}{\gh{n}(t)} - \frac{N(t,z)}{g(t)}\right| \\
&\leq& \left|\frac{N_n^{(2)}(t,z)-N(t,z)}{g(t)}\right| + \left|\frac{N_n^{(2)}(t,z)}{\gh{n}(t)}-\frac{N_n^{(2)}(t,z)}{g(t)}\right|\\
&=& \frac{1}{g(t)}\big|N_n^{(2)}(t,z)-N(t,z)\big| + N_n^{(2)}(t,z)\left|\frac{1}{\gh{n}(t)}-\frac{1}{g(t)}\right|.\ee
The first term converges uniformly to zero in probability over \A\ by (\ref{unif_conv:num_Fpi2}). The second term converges uniformly to zero in probability by Lemma \ref{lem:unif_cons_hn2_gn}, and uniform consistency of \Fpi{2}\ follows.\\ \mbox{}
\end{proof}

Each bivariate distribution function $F$ has to satisfy
\bef\label{cond:biv_cdf}
\forall\, x_1<x_2, y_1<y_2\ :\ F(x_2,y_2) - F(x_2,y_1) - F(x_1,y_2) + F(x_1,y_1) \geq 0.
\eef
This condition requires that each rectangle $[x_1,x_2]\times[y_1,y_2]$ has a nonnegative mass and suggests that some shape constraints on $F_0$ are imposed by the model. However, in Theorem \ref{th:pos_dens} below, we prove that it is not necessary to use this shape constraint to estimate $F_0$ since the estimator \Fpi{2}\ satisfies condition (\ref{cond:biv_cdf}) asymptotically. To prove this, we prove that the Lebesgue density \fpi{2}\ is positive, with probability converging to one. The estimator \Fpi{1}\ does not have a density w.r.t.\ Lebesgue measure $\lambda_2$, hence a similar result can not be proved in this way for \Fpi{1}. To prove Theorem \ref{th:pos_dens}, we need stronger conditions on \an\ and \bn\ than condition $(C.1)$.
\begin{itemize}
\item[$(C.2)$] The smoothing parameters \an\ and \bn\ converge to zero as $n\to\infty$ and satisfy
\be \lim_{n\to\infty} n\an^2\bn^2 = \infty, \ \ \lim_{n\to\infty} n\an^3\bn = \infty.\ee
\end{itemize}
Note that sequences \an\ and \bn\ satisfying condition $(C.2)$ also satisfy condition $(C.1)$ and $n\an^3\to\infty$.

\begin{theorem}\label{th:pos_dens}
Assume $F_0$ and $g$ satisfy conditions $(F.1)$ and $(G.1)$. Also assume $k$ and $\tilde k$ satisfy conditions $(K.2)$ and $(K.3)$. In addition, assume $k'$ and $\partial_1\tilde k$ are uniformly continuous. Furthermore, let \an\ and \bn\ satisfy condition $(C.2)$. Let $\cS\subset[0,\infty)^2$ be compact and such that $f_0$ is uniformly continuous on an open subset that contains \cS\ and for all $\delta>0$, $\M_\delta = \big\{(t,z)\in[0,\infty)^2 : f_0(t,z) \geq 2\delta\big\}\cap\cS$. Then for $\delta>0$,
\bef\label{pos_dens_fn2}
P\left(\forall\ (t,z)\in\M_\delta\ :\ \fpi{2}(t,z) > \frac{l_g^2}{2u_g^2}\delta \right) \conv 1,
\eef
where \fpi{2} is the Lebesgue density of \Fpi{2} and $l_g$ and $u_g$ are as defined in Lemma \ref{lem:bounds_g}.
\end{theorem}

\begin{proof}
Fix $\delta>0$. First note that since 
\be \frac{\partial^2}{\partial t\partial z} \int_0^z \hat h_{n,1}^{(2)}(t,v)\,dv = \partial_1 \hat h_{n,1}^{(2)}(t,z)\ee
we have the following expression for \fpi{2}
\bef\label{def:fn2}
\fpi{2}(t,z) = \frac{\partial^2}{\partial u\partial v} \Fpi{2}(u,v)\Big|_{(u,v)=(t,z)} = \frac{\hat g_n(t)\partial_1 \hat h_{n,1}^{(2)}(t,z) - \hat g_n'(t) \hat h_{n,1}^{(2)}(t,z)}{\hat g_n(t)^2}.
\eef

We first consider the numerator and prove that
\bef\label{pos_dens2_fn2}
P\big(\forall\ (t,z)\in\M_\delta\ :\ \hat g_n(t)\partial_1 \hat h_{n,1}^{(2)}(t,z) - \hat g_n'(t) \hat h_{n,1}^{(2)}(t,z) > 2l_g^2\delta \big) \conv 1.
\eef
For this, note that for all $(t,z)\in\M_\delta$
\be \lefteqn{\hat g_n(t)\partial_1 \hat h_{n,1}^{(2)}(t,z) - \hat g_n'(t) \hat h_{n,1}^{(2)}(t,z) = \hat g_n(t)\big(\partial_1\hat h_{n,1}^{(2)}(t,z)-\partial_1 h_1(t,z)\big) + \hat h_{n,1}^{(2)}(t,z)\big(g'(t)-\hat g_n'(t)\big)}\\
&& + \partial_1h_1(t,z)\big(\hat g_n(t)-g(t)\big) + g'(t)\big(h_1(t,z)-\hat h_{n,1}^{(2)}(t,z)\big) + g(t)\partial_1h_1(t,z) - g'(t)h_1(t,z)\\
&\geq& - \sup_{t\in proj_X\M_\delta} \hat g_n(t)\sup_{(t,z)\in\M_\delta}\big|\partial_1\hat h_{n,1}^{(2)}(t,z)-\partial_1 h_1(t,z)\big|\\
&& - \sup_{(t,z)\in\M_\delta}\hat h_{n,1}^{(2)}(t,z)\sup_{t\in proj_X\M_\delta}\big|g'(t)-\hat g_n'(t)\big|\\
&& - \sup_{(t,z)\in\M_\delta}\partial_1h_1(t,z)\sup_{t\in proj_X\M_\delta}\big|\hat g_n(t)-g(t)\big|\\
&& - \sup_{t\in proj_X\M_\delta}g'(t)\sup_{(t,z)\in\M_\delta}\big|h_1(t,z)-\hat h_{n,1}^{(2)}(t,z)\big| + g(t)\partial_1h_1(t,z) - g'(t)h_1(t,z),\ee
with $proj_X\M_\delta = \big\{t : (t,z)\in\M_\delta \mbox{ for some } z\big\}$. By Lemma \ref{lem:unif_cons_hn2_gn} all random terms converge to zero in probability. Since $g(t)\partial_1h_1(t,z) - g'(t)h_1(t,z) = g(t)^2f_0(t,z)$ we have that the last term is bounded below by $\inf_{(t,z)\in\M_\delta}g(t)^2f_0(t,z)\geq2l_g^2\delta$ by Lemma \ref{lem:bounds_g}.

By Lemma \ref{lem:bounds_g} and the uniform consistency of $\hat g_n$ [see Lemma \ref{lem:unif_cons_hn2_gn}], we have that $0<\frac12 l_g< \hat g_n(t)< 2u_g<\infty$ for all $t\in proj_X\M_\delta$ with probability converging to one. This implies that for all $(t,z)\in\M_\delta$
\be \fpi{2}(t,z) \geq \frac{\hat g_n(t)\partial_1 \hat h_{n,1}^{(2)}(t,z) - \hat g_n'(t) \hat h_{n,1}^{(2)}(t,z)}{4u_g^2} > \frac{2l_g^2}{4u_g^2}\delta,\ee
with probability converging to one. Hence (\ref{pos_dens_fn2}) follows.
\end{proof}

\begin{remark}
If, in addition to condition $(F.1)$, we assume that $f_0$ is uniformly continuous on $[0,\infty)^2$, this theorem implies that for each $\delta>0$ and $M>0$, the restriction of \Fpi{2}\ to the set $\big\{(t,z)\in[0,M]^2:f_0(t,z)\geq\delta\big\}$ will asymptotically be the restriction to this set of a bivariate distribution function $\tilde F_n$ on $[0,\infty)^2$.
\end{remark}

\section{Asymptotic distributions}\label{sec:asymp_Fpi1_Fpi2}
In this section we derive the asymptotic distribution of both plug-in inverse estimators. Although the estimator \Fpi{2}\ is more natural, we start with the estimator \Fpi{1}\ since deriving its asymptotic distribution is easier. Subsequently, we prove that for certain choices of the smoothing parameter \bn\ the estimators \Fpi{1}\ and \Fpi{2}\ are asymptotically equivalent, yielding the asymptotic distribution of \Fpi{2}.

\begin{theorem}\label{th:as_Fn1}
Assume $F_0$ and $g$ satisfy conditions $(F.1)$ and $(G.1)$. Also assume $k$ satisfies condition $(K.2)$. Fix $t_0, z_0>0$ such that $\partial_1^2F_0(t,z)$ and $g''(t)$ exist and are continuous in a neighborhood of $(t_0,z_0)$ and $t_0$, respectively, and $\partial_1^2F_0(t_0,z_0)+2g'(t_0)\partial_1F_0(t_0,z_0)/g(t_0)\not=0$ and $g(t_0)>0$. Then for $\an = cn^{-1/5}$,
\be n^{2/5}\big(\Fpi{1}(t_0,z_0) - F_0(t_0,z_0)\big) \Dconv \N(\mu_1,\sigma^2) \ee
where
\bef
\mu_1 &=& \frac12c^2m_2(k)\left\{\partial_1^2F_0(t_0,z_0)+2\frac{g'(t_0)\partial_1F_0(t_0,z_0)}{g(t_0)}\right\}, \label{def:mu_Fn1} \\
\sigma^2 &=& c^{-1}\frac{F_0(t_0,z_0)\big(1-F_0(t_0,z_0)\big)}{g(t_0)}\int k(u)^2\,du. \label{def:s_Fn1}
\eef
\end{theorem}

\begin{remark}
In case $\partial_1^2F_0(t_0,z_0)+2g'(t_0)\partial_1F_0(t_0,z_0)/g(t_0)=0$, the rate of convergence changes because the bias is of a different asymptotic order. This is in line with results for other kernel smoothers in case of vanishing first order bias terms.
\end{remark}

The proof of this theorem, a combination of the Lindeberg-Feller Central Limit Theorem and the Delta-method, is given in the Appendix.

To illustrate the pointwise asymptotic results we simulate $m=1\,000$ times a sample of size $n=5\,000$, using $F_0(x,y)=\frac12xy(x+y)$ for $x,y\in[0,1]$ and $g(t) = 2t$ for $t\in[0,1]$. For each sample we determine the estimator $\Fpi{1}(0.5,0.5)$ (using kernel density $k(y) = \frac34(1-y^2)1_{[-1,1]}(y)$ and smoothing parameter $\an=0.09$) and the resulting value of $n^{2/5}\big(\Fpi{1}(0.5,0.5)-F_0(0.5,0.5)\big)$. Figure \ref{fig:asymp_Fn1} shows these $m$ values, in a QQ-plot (with the line $y=\mu_1+x\sigma$) as well as in a histogram (with the $\N(\mu_1,\sigma^2)$ density). Here $\mu_1$ and $\sigma$ are as defined in (\ref{def:mu_Fn1}) and (\ref{def:s_Fn1}) for this $F_0$ and $g$. 

\begin{center}
{\bf [Figure \ref{fig:asymp_Fn1} here]}
\end{center}

Under definition $(K.1)$ and assumptions $(K.2)$ and $(K.3)$ on the kernel densities $k$ and $\tilde k$, we can prove that for $t_0,z_0>0$ fixed $n^{2/5}\big(\Fpi{2}(t_0,z_0)-\Fpi{1}(t_0,z_0)\big)$ converges to zero in probability whenever \bn\ converges faster to zero than $n^{-1/5}$. As a consequence, these estimators are (first order) asymptotically equivalent. For \bn\ tending to zero slower than $n^{-1/5}$, $n^{2/5}\big|\Fpi{2}(t_0,z_0)-\Fpi{1}(t_0,z_0)\big| \conv \infty$ in probability. These results are more precisely stated in Theorem \ref{th:as_equiv_Fn1Fn2} and Corollary \ref{cor:as_distr_Fn2}.

\begin{theorem}\label{th:as_equiv_Fn1Fn2}
Assume $F_0$ and $g$ satisfy conditions $(F.1)$ and $(G.1)$. Also assume $k$ and $\tilde k$ satisfy conditions $(K.2)$ and $(K.3)$. Fix $t_0, z_0>0$ such that $\partial_2^2 F_0(t,z)$ and $g(t)$ exist and are continuous in a neighborhood of $(t_0,z_0)$ and $t_0$, respectively, and $\partial_2^2 F_0(t_0,z_0)\not=0$ and $g(t_0)\not=0$. Let $\an = c_1n^{-1/5}$ and $\bn = c_2n^{-\beta}$, then for $\beta > 1/5$
\be n^{2/5}\big(\Fpi{2}(t_0,z_0) - \Fpi{1}(t_0,z_0) \big) \Pconv 0,\ee
for $\beta = 1/5$
\be n^{2/5}\big(\Fpi{2}(t_0,z_0) - \Fpi{1}(t_0,z_0) \big) \Pconv \frac12c_2^2m_2(k)\partial_2^2F_0(t_0,z_0)\ee
while for $\beta <1/5$ $n^{2/5}\big|\Fpi{2}(t_0,z_0) - \Fpi{1}(t_0,z_0) \big| \Pconv \infty$.
\end{theorem}
The proof of this theorem is given in the Appendix.

As a consequence of this theorem, the estimators \Fpi{1}\ and \Fpi{2}\ are pointwise asymptotically equivalent for $\beta > 1/5$, while for $\beta = 1/5$, $\Fpi{2}(t_0,z_0)$ has an additional (possibly negative) asymptotic bias term.

\begin{corollary}\label{cor:as_distr_Fn2}
In addition to the conditions of Theorem \ref{th:as_Fn1}, assume $\partial_2^2F_0(t_0,z_0)\not=0$ and $g(t_0)\not=0$. Let $\an = c_1n^{-1/5}$ and $\bn = c_2n^{-\beta}$. Then for $\beta > 1/5$
\be n^{2/5}\big(\Fpi{2}(t_0,z_0) - F_0(t_0,z_0)\big) \Dconv \N(\mu_1,\sigma^2) \ee
where $\mu_1$ and $\sigma^2$ are defined in (\ref{def:mu_Fn1}) and (\ref{def:s_Fn1}) (with $c=c_1$). For $\beta = 1/5$
\be n^{2/5}\big(\Fpi{2}(t_0,z_0) - F_0(t_0,z_0)\big) \Dconv \N(\mu_2,\sigma^2),\ee
where
\bef\label{def:mu_Fn2}
\mu_2 = \mu_1 + \frac12c_2^2m_2(\tilde k)\partial_2^2F_0(t_0,z_0).
\eef
\end{corollary}

\begin{proof}
This immediately follows from Theorem \ref{th:as_equiv_Fn1Fn2}.
\end{proof}

Figure \ref{fig:Rn} shows the values of $n^{2/5}\big(\Fpi{2}(0.5,0.5)-\Fpi{1}(0.5,0.5)\big)$ as a function of $n$ with $\an=\frac12n^{-1/5}$ and $\bn = \frac12 n^{-1/3}$. The solid lines are the lines $\pm \frac12n^{-1/6}$, the order of the standard deviation of $n^{2/5}\big(\Fpi{2}(0.5,0.5)-\Fpi{1}(0.5,0.5)\big)$ (see the proof of Theorem \ref{th:as_equiv_Fn1Fn2} in the appendix). For $F_0$ and $g$ we used the same setting as in Figure \ref{fig:asymp_Fn1}, for $\tilde k$ we used $\tilde k(x,y)=k(x)k(y)$ the product kernel density with $k(u) = \frac34(1-u^2)$ for $u\in[-1,1]$.

\begin{center}
{\bf [Figure \ref{fig:Rn} here]}
\end{center}

Figure \ref{fig:asymp_Fn2} shows $m=1\,000$ values of $n^{2/5}\big(\Fpi{2}(0.5,0.5)-F_0(0.5,0.5)\big)$ for $n=5\,000$, $\an=\frac12n^{-1/5}$ and $\bn=\frac12n^{-1/3}$, in a QQ-plot (with the line $y=\mu_1+x\sigma$) as well as in a histogram (with the $\N(\mu_1,\sigma^2)$ density). Here $\mu_1$ and $\sigma$ are as defined in (\ref{def:mu_Fn1}) and (\ref{def:s_Fn1}) for $F_0$, $g$ and $\tilde k$ the same as in Figure \ref{fig:Rn}.

\begin{center}
{\bf [Figure \ref{fig:asymp_Fn2} here]}
\end{center}

\section{Smooth functionals}\label{sec:smooth_functionals}
It is well known that in the current status model certain functionals of the model can be estimated at $\sqrt{n}$ rate, although the pointwise estimation rate is lower, see, e.g., \citeN{groeneboom:96}. In the continuous marks model we have a similar situation and we briefly sketch how the theory of smooth functionals applies here. In the ``hidden space'' one would be allowed to observe the random variable $(X,Y)$ with distribution function $F$, and the so-called score operator from functions on the hidden space to functions on the observation space is in this case given by
\be [L_F(a)](t,z,\delta)&=&\E\,\{a(X,Y)|(T,Z,\Delta)=(t,z,\delta)\}\\
&=&\frac{\delta\int_0^t a(x,z)\,dF_z(x)}{F_z(t)}+\frac{(1-\delta)\int_{x=t}^{\infty}\int_{y=0}^{\infty} a(x,y)\,dF_y(x)dy}{1-F(t,\infty)},
\ee
where $F_z(x)=\partial_2 F(x,z)=\frac{\partial}{\partial z}F(x,z)$. Note that the $F_z$ correspond to the component subdistribution functions in the model with finitely many competing risks and that $F(t,\infty)=\int_0^{\infty} F_z(t)\,dz$. Here $L_F$ is a mapping from $L_2^0(F)$ to $L_2^0(H)$, where $L_2^0(F)$ denotes the space of square integrable functions $a$ with zero expectation, i.e.
\bef\label{hidden_space}
\E_F\, a(X,Y)=\int a(x,y)\,dF(x,y)=0,\qquad \E_F\, a(X,Y)^2=\int a(x,y)^2\,dF(x,y)<\infty.
\eef
Similarly, $L_2^0(H)$ is the space of functions $b$ with the properties:
\be \E_H\, b(T,Z,\Delta)=\int b(t,z,\delta)\,dH(t,z,\delta)=0,\qquad \E_H\, b(T,Z,\Delta)^2=\int b(t,z,\delta)^2\,dH(t,z,\delta)<\infty.\ee
Using the first relation in (\ref{hidden_space}) we get:
\be [L_F(a)](t,z,\delta)&=&\frac{\delta\int_0^t a(x,z)\,dF_z(x)}{F_z(t)}+\frac{(1-\delta)\int_{x=t}^{\infty}\int_{y=0}^{\infty} a(x,y)\,dF_y(x)\,dy}{1-F(t,\infty)}\\
&=&\frac{\delta\int_0^t a(x,z)\,dF_z(x)}{F_z(t)}-\frac{(1-\delta)\int_{x=0}^t\int_{y=0}^{\infty} a(x,y)\,dF_y(x)\,dy}{1-F(t,\infty)}.
\ee

We now consider the adjoint of $L_F$, mapping the functions $b\in L_2^0(H)$ back into $L_2^0(F)$. The adjoint is given by:
\be [L_F^*(b)](x,y)=\int_{t=x}^{\infty} b(t,y,1)\,dG(t)+\int_{t=0}^x b(t,0,0)\,dG(t).\ee
This is analogous to what we get in the current status model, see e.g., \citeN{groeneboom:96}.

In order to make this somewhat more concrete, we consider the functional
\bef\label{mean_functional}
\mu_F=\int x\, dF_{0,X}(x)=\int x\,dF(x,\infty).
\eef
Then the score function in the hidden space is:
\be a(x,y)=x-\int x\,dF(x,\infty)=x -\iint u\,dF_w(u)\,dw,\ee
so only depends on the first argument, and we have to solve the equation
\be \int_{t=x}^{\infty} b(t,z,1)\,dG(t)+\int_{t=0}^x b(t,0,0)\,dG(t)=x -\iint u\,dF_w(u)\,dw,\ee
where $b$ has to be in the (closure of the) range of the score operator, so this would be
\be b(t,z,\delta)=\frac{\delta\int_0^t a(x,z)\,dF_z(x)}{F_z(t)}-\frac{(1-\delta)\int_{x=0}^t\int_{y=0}^{\infty} a(x,y)\,dF_y(x)\,dy}{1-F(t,\infty)}, \mbox{ for some $a$,}\ee
if $b$ is in the range itself (and not only its closure). We therefore consider the equation:
\be \int_{t=x}^{\infty}\frac{\int_{u=0}^t a(u,z)\,dF_z(u)}{F_z(t)}\,dG(t)-\int_{t=0}^x\frac{\int_{u=0}^t\int_{y=0}^{\infty} a(u,y)\,dF_y(u)\,dy}{1-F(t,\infty)}\,dG(t) =x -\iint u\,dF_w(u)\,dw.\ee
Differentiation w.r.t.\ $x$ yields:
\be -\frac{\int_{u=0}^x a(u,z)\,dF_z(u)}{F_z(x)} -\frac{\int_{u=0}^x\int_{y=0}^{\infty} a(u,y)\,dF_y(u)\,dy}{1-F(x,\infty)} =\frac1{g(x)}.\ee
Letting $\phi(x,z)=\int_{u=0}^x a(u,z)\,dF_z(u)$, this is solved by taking
\be \phi(x,z)=-\frac{F_z(x)\big(1-F(x,\infty)\big)}{g(x)}. \ee
So we get
\be b(t,z,\delta) &=& -\frac{\delta F_z(t)\big(1-F(t,\infty)\big)}{F_z(t)g(t)} + \frac{(1-\delta)\big(1-F(t,\infty)\big)\int_{y=0}^{\infty}F_y(t)\,dy}{\big(1-F(t,\infty)\big)g(t)}\\
&=&-\frac{\delta\big(1-F(t,\infty)\big)}{g(t)} + \frac{(1-\delta)F(t,\infty)}{g(t)},\ee
implying that the efficient asymptotic variance for estimating the mean functional $\mu_F$, defined by (\ref{mean_functional}), is given by:
\bef\label{information_lowerbound}
\int b(t,z,\delta)^2\,dH(t,z,\delta)=\int\frac{F(t,\infty)\big(1-F(t,\infty)\big)}{g(t)}\,dt,
\eef
which (not surprisingly) is the same expression as one gets in the current status model.

The next question becomes whether taking $\int x\,d\hat F_n(x,\infty)$, where $\hat F_n$ is one of our proposed estimators, will lead to an efficient estimate of $\mu_F$, in the sense that it converges at rate $\sqrt{n}$, with an asymptotic variance which attains the information lower bound (\ref{information_lowerbound}).

Let us consider the estimator, defined by (\ref{def:Fn_one}), and more specifically, the estimator obtained by taking $k(y)=\tfrac121_{[-1,1]}(y)$. Then (\ref{def:Fn_one}) becomes
\be \Fpi{1}(x,z)=\frac{\int_{u\in[x-\an,x+\an],\,y\in(0,z]}d\Hn{n}(u,y,1)}{\int_{u\in[x-\an,x+\an],\,y\geq0}d\Hn{n}(u,y,\delta)},\ee
where \Hn{n}\ is the empirical distribution of the sample $W_1, \ldots, W_n$. Also assume that $f$ has compact support, say $[0,1]^2$, as in the setting of Figure \ref{fig:asymp_Fn1}. Then we get as the estimate of $F_{0,X}$:
\be \Fpi{1}(x,1)=\Fpi{1}(x,\infty)=\frac{\int_{u\in[x-\an,x+\an],\,y>0}d\Hn{n}(u,y,1)}{\int_{u\in[x-\an,x+\an],\,y\geq0}d\Hn{n}(u,y,\delta)}.\ee

To see whether this estimator leads to an efficient estimate of $\mu_F$, we have to perform a bias-variance analysis. We first consider the bias. Let $F_{\an}$ be defined by
\be F_{\an}(x)=\frac{\int_{u\in[x-\an,x+\an],\,y>0}dH_{F_0}(u,y,1)}{\int_{u\in[x-\an,x+\an],\,y\geq0}dH_{F_0}(u,y,\delta)},\ee
where $H_{F_0}$ is the distribution function of $(T,Z,\Delta)$ in the observation space. Then
\be \lefteqn{ \int x\,dF_{\an}(x)=\int_0^1\big(1-F_{\an}(x)\big)\,dx =\int_0^1\frac{\int_{u\in[x-\an,x+\an]}dH_{F_0}(u,0,0)}{\int_{u\in[x-\an,x+\an],\,y\geq0}dH_{F_0}(u,y,\delta)}\,dx}\\
&=&\int_{u=-\an}^{\an}\int_{x\in[0,u+\an]}\frac{g(u)\big(1-F_0(u,1)\big)}{\int_{x-\an}^{x+\an}g(v)\,dv}\,dx\,du
+\int_{u=\an}^{1-\an}\int_{x\in[u-\an,u+\an]}\frac{g(u)\big(1-F_0(u,1)\big)}{\int_{x-\an}^{x+\an}g(v)\,dv}\,dx\,du\\
&&\qquad+\int_{u=1-\an}^{1+\an}\int_{x\in[u-\an,1]}\frac{g(u)\big(1-F_0(u,1)\big)}{\int_{x-\an}^{x+\an}g(v)\,dv}\,dx\,du \ee
We have, if $g$ is twice continuously differentiable and stays away from zero on $[0,1]$
\be &&\hspace{-2cm}\int_{x\in[u-\an,u+\an]}\frac1{G(x+\an)-G(x-\an)}\,dx=\int_{x\in[u-\an,u+\an]}\frac1{2\an g(x)+\tfrac16g''(x)\an^3+\dots}\,dx\\
&=&\int_{x\in[u-\an,u+\an]}\frac1{2\an g(x)\big(1+O(\an^2)\big)}\,dx =\frac1{g(u)}+O\left(\an^2\right), \ee
and hence
\be\int_{u=-\an}^{\an}\int_{x\in[0,u+\an]}\frac{g(u)\big(1-F_0(u,1)\big)}{\int_{x-\an}^{x+\an}g(v)\,dv}\,dx\,du=\int_{u=0}^{\an}\big(1-F_0(u,1)\big)\,du+O\left(\an^2\right).\ee
We also have
\be\int_{u=\an}^{1-\an}\int_{x\in[u-\an,u+\an]}\frac{g(u)\big(1-F_0(u,1)\big)}{\int_{x-\an}^{x+\an}g(v)\,dv}\,dx\,du=\int_{u=\an}^{1-\an}\big(1-F_0(u,1)\big)\,du+O\left(\an^2\right),\ee
and similarly
\be\int_{u=1-\an}^{1+\an}\int_{x\in[u-\an,1]}\frac{g(u)\big(1-F_0(u,1)\big)}{\int_{x-\an}^{x+\an}g(v)\,dv}\,dx\,du=\int_{u=1-\an}^1\big(1-F_0(u,1)\big)\,du+O\left(\an^2\right).\ee

So we obtain
\bef\label{bias_term}
\int_0^1\big(1-F_{\an}(x)\big)\,dx=\int_0^1\big(1-F_0(x,1)\big)\,dx+O\left(\an^2\right).
\eef
Empirical process methods give us
\bef\label{var_term}
\int\big(\Fpi{1}(x,1)-F_{\an}(x)\big)\,dx=O_p\left(n^{-1/2}\right).
\eef
So (\ref{bias_term}) and (\ref{var_term}) give us that, if (for example) $\an$ is of order $n^{-1/3}$,
\bef\label{sqrt{n}_order}
\int\big(\Fpi{1}(x,1)-F_0(x,1)\big)\,dx=O_p\left(n^{-1/2}\right).
\eef
Note that this does not follow if $\an$ is of order $n^{-1/5}$, since the bias term is too large in that case!

For the asymptotic variance, one has to analyze:
\be\int_{x=0}^1\left\{\frac{\int_{u\in[x-\an,x+\an]}d\Hn{n}(u,0,0)}{\int_{u\in[x-\an,x+\an],\,y\in[0,1]}d\Hn{n}(u,y,\delta)}-\frac{\int_{u\in[x-\an,x+\an]}dH_{F_0}(u,0,0)}{\int_{u\in[x-\an,x+\an],\,y\in[0,1]}dH_{F_0}(u,y,\delta)}\right\}\,dx,\ee
which can be written as
\be\lefteqn{\int_{x=0}^1\frac{\int_{u\in[x-\an,x+\an]}d\left(\Hn{n}-H_{F_0}\right)(u,0,0)}{\int_{u\in[x-\an,x+\an],\,y\in[0,1]}d\Hn{n}(u,y,\delta)}\,dx}\\
&&\hspace{2cm}-\int_{x=0}^1\frac{\int_{u\in[x-\an,x+\an],\,y\in[0,1]}d\left(\Hn{n}-H_{F_0}\right)(u,y,\delta)\int_{u\in[x-\an,x+\an]}dH_{F_0}(u,0,0)}{\int_{u\in[x-\an,x+\an],\,y\in[0,1]}d\Hn{n}(u,y,\delta)\int_{u\in[x-\an,x+\an],\,y\in[0,1]}dH_{F_0}(u,y,\delta)}\,dx\\
&\sim&\int_{x=0}^1\frac{F_0(x,1)\int_{u\in[x-\an,x+\an]}d\left(\Hn{n}-H_{F_0}\right)(u,0,0)}{2g(x)\an}\,dx\\
&&\hspace{2cm}-\int_{x=0}^1\frac{\big(1-F_0(x,1)\big)\int_{u\in[x-\an,x+\an],\,y\in(0,1]}d\left(\Hn{n}-H_{F_0}\right)(u,y,1)}{2g(x)\an}\,dx\\
&\sim&\int_{u\in[0,1]}\frac{F_0(u,1)}{g(u)}\,d\left(\Hn{n}-H_{F_0}\right)(u,0,0)-\int_{u\in[0,1],\,y\in(0,1]}\frac{1-F_0(u,1)}{g(u)}\,d\left(\Hn{n}-H_{F_0}\right)(u,y,1).\ee
So the asymptotic variance is given by:
\be \lefteqn{\int_{u\in[0,1]}\frac{F_0(u,1)^2}{g(u)^2}\,dH_{F_0}(u,0,0)+\int_{u\in[0,1],\,y\in(0,1]}\frac{\big(1-F_0(u,1)\big)^2}{g(u)^2}\,dH_{F_0}(u,y,1)}\\
&=&\int_0^1\frac{F_0(u,1)^2\big(1-F_0(u,1)\big)}{g(u)}\,du+\int_0^1\frac{\big(1-F_0(u,1)\big)^2F_0(u,1)}{g(u)}\,du\\
&=&\int_0^1\frac{F_0(u,1)\big(1-F_0(u,1)\big)}{g(u)}\,du=\int_0^1\frac{F_0(u,\infty)\big(1-F_0(u,\infty)\big)}{g(u)}\,du.\ee

The conclusion is that in this example, our estimator of $\mu_F$ converges at rate $\sqrt{n}$ and that its asymptotic variance attains the information lower bound, \emph{provided the bandwidth $\an$ tends to zero faster than $n^{-1/4}$}. It also illustrates that a bandwidth of order $n^{-1/5}$, which is an obvious choice for the pointwise estimation, is not suitable if we want to estimate smooth functionals, a phenomenon that seems (more or less) well known. Similar analyses can be performed for other smooth functionals, but since the local estimation problem is the main focus of our paper, we will not pursue this further here.

\section{Simulation study}\label{sec:sim_study}
The estimators \Fpi{1}\ and \Fpi{2}\ are asymptotically equivalent for sufficiently small choices of the smoothing parameter $\bn$. To get some insight in the finite sample differences between the estimators, we run a simulation study. We simulated data according to $F_0(x,y)=\frac12xy(x+y)$ for $x,y\in[0,1]$ and $g(t) = 2t$ for $t\in[0,1]$ for different sample sizes $n=500$, $n=1\,000$, $n=5\,000$ and $n=10\,000$. For each simulation we computed the estimators $\Fpi{1}(t_0,z_0)$ and $\Fpi{2}(t_0,z_0)$ for two different values of $(t_0,z_0)$ and different values of the smoothing parameters $\an$ and $\bn$. We repeated this $B=250$ times, resulting in $250$ estimates $\hat F_{n,\an,\bn}^{(i),1}(t_0,z_0), \hat F_{n,\an,\bn}^{(i),2}(t_0,z_0), \ldots, \hat F_{n,\an,\bn}^{(i),250}(t_0,z_0)$ ($i=1,2$) for each value of the smoothing parameters $\an$ and $\bn$. Then, we estimated the Mean Squared Error (MSE) of the estimator $\Fpi{i}(t_0,z_0)$ by
\be \frac1B\sum_{j=1}^B \left(\hat F_{n,\an,\bn}^{(i),j}(t_0,z_0)-F_0(t_0,z_0)\right)^2.\ee

Table \ref{table:sim_study} shows the minimum value of the estimated MSE for each estimator, for each $n$ and in two different points $(t_0,z_0)$. It also shows the values of the smoothing parameters $\an$ and $\bn$ that yielded this value. The standard error of the mean of the squared differences $\big(\hat F_{n,\an,\bn}^{(i),j}(t_0,z_0)-F_0(t_0,z_0)\big)^2$ are given in brackets. The binned MLE $\tilde F_n$ studied by \citeN{maathuis:08} and the Maximum Smoothed Likelihood Estimator (MSLE) \FMS\ studied by \shortciteN{witte:10} are included in this simulation study.
\begin{center}
{\bf [Table \ref{table:sim_study} here]}
\end{center}

Figure \ref{fig:sim_study} shows the resulting values of estimated MSEs as function of $\an$. For $\tilde F_n$, the smoothing parameter \an\ is the binwidth in $z$-direction, for \FMS\ we have that \an\ and \bn\ are the bindwidths in $t$- and $z$-direction, respectively. Both \Fpi{2}\ and \FMS\ depend on two smoothing parameters, and we fixed the value of $\bn$ to be equal to that value that yielded the overall minimal estimated MSEs of the estimators. Determining the optimal value(s) of the smoothing parameter(s) for $\tilde F_n$ and \FMS\ was a bit tedious; the estimated MSE of $\tilde F_n$ was very wiggly, the estimated MSE of \FMS\ is only nicely $U$-shaped for bigger values of $n$ due to computational issues. Although we choose the values of \an\ and \bn\ also as the minimizing binwidths of the estimated MSEs, these choices might not be good estimates.
\begin{center}
{\bf [Figure \ref{fig:sim_study} here]}
\end{center}

This simulation study, of which only some results are illustrated in Figure \ref{fig:sim_study} for \Fpi{1}\ and \Fpi{2}\ only, shows that the estimated MSEs of both plug-in inverse estimators are almost equal. Based on the estimated MSEs and the standard errors of the mean of the squared differences between the estimators \Fpi{1}\ and \Fpi{2}\ and the true distribution function, confidence intervals can be computed. The intervals for \Fpi{1}\ and \Fpi{2}\ have non-empty intersections, implying that for this specific example there is no significant finite sample difference between the smooth plug-in inverse estimators.

\section{Bandwidth selection in practice}\label{sec:sboot}
The estimators \Fpi{1}\ and \Fpi{2}\ depend on smoothing parameters \an\ and \bn\ (only \Fpi{2}). As with usual kernel density estimators, the estimators are quite sensitive to the choice of the smoothing parameters. Small values of \an\ and \bn\ will result in wiggly estimators reflecting the high variance, whereas big values of \an\ and \bn\ will give smooth stable, but biased, estimators. One way to obtain good smoothing parameters that depend on the data is via the smoothed bootstrap.

The focus of this paper is on the pointwise asymptotic behavior of the estimators \Fpi{1}\ and \Fpi{2}, so also the choice of \an\ and \bn\ is only considered locally at the point $(t_0,z_0)$. The smoothed bootstrap differs from the empirical bootstrap in the distribution it samples from. In the empirical bootstrap one samples from the empirical distribution function of the data, whereas in the smoothed bootstrap one samples from a usually slightly oversmoothed estimator for the observation density $h_{F_0}$.

We now describe this method more specifically in our model. Let $\hat g_{n,\alpha_0}=:\hat g_0$ and $\hat F_{n,\alpha_0,\beta_0}^{(2)}=:\hat F_0$ be the kernel estimator and the smooth plug-in inverse estimator for $g$ and $F_0$, respectively, with smoothing parameters $\alpha_0$ and $\beta_0$. Then, $(X_1^{*,1},Y_1^{*,1}),(X_2^{*,1},Y_2^{*,1}),\ldots, (X_n^{*,1},Y_n^{*,1})$ are sampled from $\hat F_0$, $T_1^{*,1}, T_2^{*,1},\ldots,T_n^{*,1}$ from $\hat g_0$ independently of $(X_i^{*,1},Y_i^{*,1})$. The variables $\Delta_i^{*,1}$ and $Z_i^{*,1}$ are defined as $1_{\{X_i^{*,1}\leq T_i^{*,1}\}}$ and $Y_i^{*,1}\cdot\Delta_i^{*,1}$, respectively. The estimators $\hat F_{n,\an,1}^{(1)}$ and $\hat F_{n,\an,\bn,1}^{(2)}$ are determined at the point $(t_0,z_0)$ for several values of \an\ and \bn\ based on the sample $(T_1^{*,1},Z_1^{*,1},\Delta_1^{*,1}), \ldots ,(T_n^{*,1},Z_n^{*,1},\Delta_n^{*,1})$. Note that now we make the dependence of the estimators on \an\ and \bn\ explicit in the notation of the estimators. Actually, we only need the precise values for those observations $(T_i^{*,1},Z_i^{*,1},\Delta_i^{*,1})$ that fall in $[t_0-\an,t_0+\an]\times[0,z_0+\bn]\times\{0,1\}$, the precise values of $T_i^{*,1}$ for those observations that fall in $[t_0-\an,t_0+\an]\times(z_0+\bn,\infty]\times\{1\}$ and the numbers of observations in the various regions outside these areas (rather than their exact locations) to compute $\hat F_{n,\an,1}^{(1)}(t_0,z_0)$ and $\hat F_{n,\an,\bn,1}^{(2)}(t_0,z_0)$. Hence, only on this strip monotonicity of $\hat F_0$ is needed as well as positivity of $\hat F_{n,\an,1}^{(1)}(\infty,\infty)-\hat F_{n,\an,1}^{(1)}(t_0+\an,\infty)$ and $\hat F_{n,\an,\bn,1}^{(2)}(\infty,\infty) -\hat F_{n,\an,\bn,1}^{(2)}(t_0+\an,\infty)$.

The procedure described above is repeated $B$ times resulting in $B$ estimators $\hat F_{n,\an,1}^{(1)}, \ldots, \hat F_{n,\an,B}^{(1)}$ and $\hat F_{n,\an,\bn,1}^{(2)}, \ldots, \hat F_{n,\an,\bn,B}^{(2)}$. Then, the MSEs of $\hat F_{n,\an}^{(1)}(t_0,z_0)$ and $\hat F_{n,\an,\bn}^{(2)}(t_0,z_0)$ can be estimated by
\be \widehat{MSE}^{(1)}(\an;t_0,z_0) &=& \frac1B\sum_{b=1}^B\left(\hat F_{n,\an,b}^{(1)}(t_0,z_0) - \hat F_0(t_0,z_0)\right)^2,\\
\widehat{MSE}^{(2)}(\an,\bn;t_0,z_0) &=& \frac1B\sum_{b=1}^B\left(\hat F_{n,\an,\bn,b}^{(2)}(t_0,z_0) - \hat F_0(t_0,z_0)\right)^2.\ee
Then, choose those values of \an\ and \bn\ that minimize $\widehat{MSE}^{(1)}(\an;t_0,z_0)$ and $\widehat{MSE}^{(2)}(\an,\bn;t_0,z_0)$ as smoothing parameters for the estimators $\Fpi{1}(t_0,z_0)$ and $\Fpi{2}(t_0,z_0)$, respectively.

Figure \ref{fig:bw_sel} shows the estimated MSEs for a small simulation study. In this study, we took $n=100$, $B=500$, $\alpha_0=\beta_0=0.4$, $t_0=z_0=0.5$ and $F_0$ and $g$ as in Section \ref{sec:sim_study}. It also shows
\be \widetilde{MSE}^{(1)}(\an;t_0,z_0) &=& \frac1B\sum_{b=1}^B\left(\hat F_{n,\an,b}^{(1)}(t_0,z_0) - F_0(t_0,z_0)\right)^2,\\
\widetilde{MSE}^{(2)}(\an,\bn;t_0,z_0) &=& \frac1B\sum_{b=1}^B\left(\hat F_{n,\an,\bn,b}^{(2)}(t_0,z_0) - F_0(t_0,z_0)\right)^2,\ee
as function of \an. For $\widehat{MSE}^{(2)}$ and $\widetilde{MSE}^{(2)}$ it only shows the estimates for that value of \bn\ that has the smallest estimated MSE.
\begin{center}
{\bf [Figure \ref{fig:bw_sel} here]}
\end{center}

There are other methods to obtain data-dependent bandwidths, for example via cross-validation \cite{rudemo:82}. Usually in cross-validation methods a global risk measure is minimized (like the Integrated MSE), hence its minimizer can be used as a global optimal bandwidth.

\section{Concluding remarks}
In this paper we consider two plug-in inverse estimators for the distribution function of the vector $(X,Y)$ in the current status continuous mark model. The first estimator \Fpi{1}\ is shown to be consistent and pointwise asymptotically normally distributed. However, \Fpi{1} does not have a Lebesgue density, since it only puts mass on the lines $[0,\infty)\times \{Z_i\}$ with $Z_i>0$ for $i=1,\ldots,n$. The second estimator, \Fpi{2}, does have a Lebesgue density. For a range of possible choices of the bandwidths \an\ and \bn\, we establish consistency of this estimator. Taking $\an=n^{-1/5}$ and $\bn=n^{-\beta}$, we prove that asymptotically for $\beta < 3/10$ the Lebesgue density of \Fpi{2}\ is positive on a region where $f_0$ is positive which stays away from the boundary of its support. This means that, although for finite sample size $n$ the estimator \Fpi{2}\ need not be a bivariate distribution function, ``isotonisation'' of it is not necessary asymptotically. Put differently, any common shape regularized version of our estimator is asymptotically equivalent with our estimator. However, this only holds asymptotically, and for finite sample size $n$ it might be desirable to have an estimator which is a true bivariate distribution function, satisfying condition (\ref{cond:biv_cdf}). For example when one wants to sample in a smoothed bootstrap procedure. Furthermore, we prove that \Fpi{2}\ is asymptotically normally distributed for $\beta\geq1/5$. Hence, for $\beta\in[1/5,3/10)$, the estimator \Fpi{2}\ asymptotically behaves as a distribution function with pointwise normal limiting distribution on a large subset of $[0,\infty)^2$.

\section*{Acknowledgements}
We would like to thank two anonymous referees for their valuable comments concerning the practical use of our results and the readability of the paper.

\noindent B.I. Witte, VU University Medical Center, Department of Epidemiology and Biostatistics, PO Box 7057, 1007 BM Amsterdam, The Netherlands\\
E-mail: B.Witte@vumc.nl

\appendix
\section{Technical lemmas and proofs}
\begin{lemma}\label{lem:prop_projXMd}
Assume that $F_0$ and $g$ satisfy conditions $(F.1)$ and $(G.1)$. Let \cS\ and $\M_\delta$ be as defined in Theorem \ref{th:pos_dens}. Then $proj_X\M_\delta = \big\{t:(t,z)\in\M_\delta \text{ for some } z\big\}$ is a closed subset of $\cS_{0,X}^\circ$.
\end{lemma}
\begin{proof}
Fix $\delta>0$. To prove that $proj_X\M_\delta$ is a closed subset of $\cS_{0,X}^\circ$ we prove that
\begin{enumerate}
\item[$(i)$] $\M_\delta$ is closed in $[0,\infty)^2$,
\item[$(ii)$] $proj_X\M_\delta$ is closed in $[0,\infty)$,
\item[$(iii)$] $proj_X\M_\delta$ is a subset of $\cS_{0,X}^\circ$.
\end{enumerate}
We now start with proving $(i)$. By definition of \cS, there exists an open set $\U\supset\cS$ on which $f_0$ in uniformly continuous. Define
\be \U_\delta = \big\{ (t,z)\in\U\ :\ f_0(t,z)\geq 2\delta\} = f_0^{-1}\big[[2\delta,\infty)\big].\ee
The function $f_0$ is continuous on \U, hence $\U_\delta$ is closed. Since we also have that
\be \M_\delta = \big\{ (t,z)\in[0,\infty)^2\ :\ f_0(t,z)\geq 2\delta\} \cap\cS = \U_\delta \cap\cS,\ee
$\M_\delta$ is the intersection of two closed sets, hence closed itself.

For proving $(ii)$, assume $proj_X \M_\delta$ is not closed. Then, there exists a sequence $(x_n)_n \in proj_X \M_\delta$ with $x_n \to x\notin proj_X\M_\delta$. By $(i)$, the set $\M_\delta$ is closed, hence by definition of $proj_X\M_\delta$ there exists a sequence  $(x_n,y_n)_n \in\M_\delta$. By compactness of $\M_\delta$ (this follows from $(i)$), there exists a subsequence $(n_k)_k$ and $(x,y)\in\M_\delta$ such that $(x_{n_k},y_{n_k})_k \to (x,y)$. From this it follows that $x\in proj_X\M_\delta$. This yields a contradiction, hence $proj_X\M_\delta$ is closed.

To prove $(iii)$, first note that by uniform continuity of $f_0$ on $\M_\delta$
\be \exists\ \eta>0 \mbox{ such that } \forall\ (t,z),(s,y)\in\M_\delta \ :\ \|(t,z)-(s,y)\|<\eta \Longrightarrow |f_0(t,z)-f_0(s,y)|<\delta.\ee
Now take $t\in proj_X\M_\delta$. Then for all $s$ in a small neighborhood of $t$ and $z_s>0$ such that $(s,z_s)\in\M_\delta$
\be f_{0,X}(s) = \int_0^\infty f_0(s,z)\,dz \geq \int_{[z_s,z_s+\eta/2]}f_0(s,z) \geq \delta\eta/2,\ee
hence $t\in\cS_{0,X}^\circ$.
\end{proof}

\begin{lemma}\label{lem:bounds_g}
Assume that $F_0$ and $g$ satisfy conditions $(F.1)$ and $(G.1)$. Let \cS\ and $\M_\delta$ be as defined in Theorem \ref{th:pos_dens}. Then 
\bef\label{bounds_g}
\exists\ l_g>0, u_g<\infty \ :\ l_g\leq g(t)\leq u_g \mbox{ for all $t\in proj_X \M_\delta$}.
\eef
\end{lemma}
\begin{proof}
The set $proj_X\M_\delta$ is a closed subset of $\cS_{0,X}^\circ$ by Lemma \ref{lem:prop_projXMd}. On $\cS_{0,X}^\circ$, we have that $0<g<\infty$, hence also on $proj_X\M_\delta$. Now assume (\ref{bounds_g}) does not hold. Then, there exists a sequence $(t_n)_n\to t \in proj_X\M_\delta$ such that $g(t_n)\to0$. By the uniform continuity of $g$, this implies that $g(t)=0$, yielding a contradiction.

The existence of $u_g$ follows immediately from $(G.1)$.
\end{proof}

\begin{lemma}\label{lem:unif_cons_hn2_gn}
Assume $F_0$ and $g$ satisfy conditions $(F.1)$ and $(G.1)$. Also assume $k$ and $\tilde k$ satisfy conditions $(K.2)$ and $(K.3)$. In addition, assume $k'$ and $\partial_1\tilde k$ are uniformly continuous. Furthermore, let \an\ and \bn\ satisfy condition $(C.2)$. Let $\cS\subset[0,\infty)^2$ be compact and such that $f_0$ is uniformly continuous on an open subset that contains \cS\ and define $\norm{f}{\cS,\infty} = \sup_{(x,y)\in\cS}|f(x,y)|$ and $\cS_X = proj_X \cS$. Then
\bef
&&\norm{\hat g_n - g}{\cS_X,\infty} \Pconv 0,\ \ \norm{\hat g_n' - g'}{\cS_X,\infty} \Pconv 0 \label{unif_cons_gn}\\
&&\norm{\hat h_{n,1}^{(2)} - h_1}{\cS,\infty} \Pconv 0,\ \ \norm{\partial_1\hat h_{n,1}^{(2)} - \partial_1h_1}{\cS,\infty} \Pconv 0 \label{unif_cons_hn2}
\eef
\end{lemma}
\begin{proof}
The results in (\ref{unif_cons_gn}) follow directly from Theorem A and C in \citeN{silverman:78}. The first result in (\ref{unif_cons_hn2}) follows from Theorem 3.3 in \citeN{cacoullos:64}. By Theorem 3 in \shortciteN{mokkadem:05},
\be \lim_{n\to\infty}\big(n\an^3\bn\big)^{-1}\log P\big(\norm{\partial_1\hat h_n - \partial_1\hf}{\cS,\infty}\geq\delta\big) = -c,\ee
for some constant $c>0$ only depending on $\delta$ and $\partial_1\hf$. Hence, for $n$ sufficiently large
\be P\big(\norm{\partial_1\hat h_n - \partial_1\hf}{\cS,\infty}\geq\delta\big) \leq 2e^{-n\an^3\bn c} \conv 0.\ee

The results in \citeN{cacoullos:64} and \shortciteN{mokkadem:05} hold for density estimators, whereas the estimator $\hat h_{n,1}^{(2)}$ is a sub-density. However, the results in (\ref{unif_cons_hn2}) follow from these after defining a binomially distributed sample size $N_1$ and reason similarly as the proof of (A.3) in \shortciteN{witte:10}.
\end{proof}

\begin{proof2}{Theorem \ref{th:as_Fn1}}
Define
\be Y_i = \left(\begin{array}{c}Y_{i;1}\\Y_{i;2}\end{array}\right) = n^{-3/5}\left(\begin{array}{c}
k_{\an}\big(t_0-T_i\big)\\
1_{[0,z_0]}(Z_i)\Delta_ik_{\an}\big(t_0-T_i\big)\\
\end{array}\right).\ee
By the assumptions on $F_0$ and $g$ and condition $(K.2)$, we have
\be &&\E\,Y_{i;1} = n^{-3/5}g(t_0) + \frac12n^{-1}c^2m_2(k)g''(t_0) + \o(n^{-1}),\\
&&\Var\,Y_{i;1} = n^{-1}c^{-1}g(t_0)\int k^2(y)\,dy + \O(n^{-6/5}),\\
&&\E\,Y_{i;2} = n^{-3/5}F_0(t_0,z_0)g(t_0) + \frac12n^{-1}c^2m_2(k)\partial_1^2\big\{F_0(t_0,z_0)g(t_0)\big\} + \o(n^{-1}),\\
&&\Var\,Y_{i;2} = n^{-1}c^{-1}F_0(t_0,z_0)g(t_0)\int k^2(y)\,dy + \O(n^{-6/5}).\ee
Furthermore we have
\be \cov\,(Y_{i;1},Y_{i;2}) = n^{-1}c^{-1}F_0(t_0,z_0)g(t_0)\int k^2(y)\,dy + \O(n^{-6/5}),\ee
so that
\be
&&\sum_{i=1}^n\E\,Y_i = n^{2/5}\left(\begin{array}{c}
g(t_0)\\
F_0(t_0,z_0)g(t_0)\\
\end{array}\right) + \left(\begin{array}{c}
\frac{1}{2}c^2m_2(k)g''(t_0)\\
\frac{1}{2}c^2m_2(k)\partial_1^2\big\{F_0(t_0,z_0)g(t_0)\big\}\\
\end{array}\right), + \o(1)\\
&&\sum_{i=1}^n\Var\,Y_i = c^{-1}g(t_0)\int k(u)^2\,du \left(\begin{array}{cc}
1 & F_0(t_0,z_0)\\
F_0(t_0,z_0) & F_0(t_0,z_0)
\end{array}\right)+\O(n^{-1/5}) = \Sigma_1 + \O(n^{-1/5}).
\ee
Here we denote by $\Var\,Y_i$ the covariance matrix of the vector $Y_i$. By the Lindeberg--Feller central limit theorem we then get
\bef\label{as_distr_vec}
\lefteqn{n^{2/5}\left(\left(\begin{array}{l} \frac{1}{n}\sum_{i=1}^nk_{\an}\big(t_0-T_i\big)\\
\frac{1}{n}\sum_{i=1}^n1_{[0,z_0]}(Z_i)\Delta_ik_{\an}\big(t_0-T_i\big)\\
\end{array}\right) - \left(\begin{array}{c} g(t_0) \\
F_0(t_0,z_0)g(t_0) \end{array}\right)\right)}\nonumber \\
&& \ \ - \frac12c^2m_2(k)\left(\begin{array}{c} g''(t_0) \\ \partial_1^2\big\{F_0(t_0,z_0)g(t_0)\big\} \end{array}\right) = \sum_{i=1}^n\big(Y_i-\E\,Y_i\big) + \o(1) \Dconv \N(0,\Sigma_1).
\eef

For the pointwise asymptotic result of \Fpi{1}, note that 
\be \Fpi{1}(t_0,z_0) = \phi\left(\frac{1}{n}\sum_{i=1}^nk_{\an}\big(t_0-T_i\big),\frac{1}{n}\sum_{i=1}^n1_{[0,z_0]}(Z_i)\Delta_i k_{\an}\big(t_0-T_i\big)\right),\ee
for $\phi(u,v) = v/u$. Now applying the Delta-method to (\ref{as_distr_vec}) gives
\be n^{2/5}\big(\Fpi{1}(t_0,z_0) - F_0(t_0,z_0)\big) \Dconv \N(\mu_1,\sigma^2) \ee
where $\mu_1$ and $\sigma^2$ are defined in (\ref{def:mu_Fn1}) and (\ref{def:s_Fn1}).
\end{proof2}

\begin{proof2}{Theorem \ref{th:as_equiv_Fn1Fn2}}
For $i=1,2$, let $N_n^{(i)}(t_0,z_0)$ be the numerator in the definitions (\ref{def:Fn_one}) and (\ref{def:Fn_two}) of $\Fpi{i}(t_0,z_0)$ at a fixed point $(t_0,z_0)$, and note that we can write
\be N_n^{(2)}(t_0,z_0) &=& \frac1n\sum_i 1_{[0,z_0-\bn]}(Z_i)\Delta_i\int_0^{z_0}\tilde k_{\an,\bn}(t_0-T_i,z-Z_i)\,dz \\
&& \ \ + \frac1n\sum_i 1_{(z_0-\bn,z_0]}(Z_i)\Delta_i\int_0^{z_0}\tilde k_{\an,\bn}(t_0-T_i,z-Z_i)\,dz\\
&& \ \ + \frac1n\sum_i 1_{(z_0,z_0+\bn]}(Z_i)\Delta_i\int_0^{z_0}\tilde k_{\an,\bn}(t_0-T_i,z-Z_i)\,dz\\
&=& \frac1n\sum_i 1_{[0,z_0-\bn]}(Z_i)\Delta_i\int_0^{z_0}\tilde k_{\an,\bn}(t_0-T_i,z-Z_i)\,dz \\
&& \ \ + \frac1n\sum_i 1_{(z_0-\bn,z_0]}(Z_i)\Delta_i\int_0^{Z_i+\bn}\tilde k_{\an,\bn}(t_0-T_i,z-Z_i)\,dz\\
&& \ \ - \frac1n\sum_i 1_{(z_0-\bn,z_0]}(Z_i)\Delta_i\int_{z_0}^{Z_i+\bn}\tilde k_{\an,\bn}(t_0-T_i,z-Z_i)\,dz\\
&& \ \ + \frac1n\sum_i 1_{(z_0,z_0+\bn]}(Z_i)\Delta_i\int_0^{z_0}\tilde k_{\an,\bn}(t_0-T_i,z-Z_i)\,dz\\
&=& \frac1n\sum_i 1_{[0,z_0]}(Z_i)\Delta_ik_{\an}(t_0-T_i,z_0) - \frac1n\sum_i 1_{(z_0-\bn,z_0]}(Z_i)\Delta_i\int_{z_0}^{Z_i+\bn}\tilde k_{\an,\bn}(t_0-T_i,z-Z_i)\,dz \\
&& \ \ + \frac1n\sum_i 1_{(z_0,z_0+\bn]}(Z_i)\Delta_i\int_{Z_i-\bn}^{z_0}\tilde k_{\an,\bn}(t_0-T_i,z-Z_i)\,dz.\ee
In the last equality we use $(K.1)$, so that 
\be R_n &=& N_n^{(2)}(t_0,z_0) - N_n^{(1)}(t_0,z_0) = \frac1n\sum_{i=1}^n 1_{(z_0,z_0+\bn]}(Z_i)\Delta_i\int_{Z_i-\bn}^{z_0}\tilde k_{\an,\bn}(t_0-T_i,z-Z_i)\,dz\\
&& \ \ - \frac1n\sum_{i=1}^n 1_{(z_0-\bn,z_0]}(Z_i)\Delta_i\int_{z_0}^{Z_i+\bn}\tilde k_{\an,\bn}(t_0-T_i,z-Z_i)\,dz =: \frac1n \sum_{i=1}^n U_i.\ee

First we consider the variance of $n^{2/5}R_n$. Observe that
\be |U_i| &\leq& 1_{(z_0,z_0+\bn]}(Z_i)\Delta_i\int_{Z_i-\bn}^{z_0}\tilde k_{\an,\bn}(t_0-T_i,z-Z_i)\,dz\\
&& \ \ + 1_{(z_0-\bn,z_0]}(Z_i)\Delta_i\int_{z_0}^{Z_i+\bn}\tilde k_{\an,\bn}(t_0-T_i,z-Z_i)\,dz\\
&\leq& 1_{(z_0,z_0+\bn]}(Z_i)\Delta_i\int_{Z_i-\bn}^{Z_i+\bn}\tilde k_{\an,\bn}(t_0-T_i,z-Z_i)\,dz\\
&& \ \ + 1_{(z_0-\bn,z_0]}(Z_i)\Delta_i\int_{Z_i-\bn}^{Z_i+\bn}\tilde k_{\an,\bn}(t_0-T_i,z-Z_i)\,dz\\
&=& 1_{(z_0-\bn,z_0+\bn]}(Z_i)\Delta_ik_{\an}(t_0-T_i) := S_i,\ee
with
\be\E\,S_i^2 = \int_u\int_v1_{(z_0-\bn,z_0+\bn]}(v)k_{\an}^2(t_0-u)h_1(u,v)\,dv\,du = \an^{-1}\bn 2h_1(t_0,z_0)\int k^2(x)\,dx + \O(\bn).\ee
Since
\be \Var\,R_n = \frac1n\Var\,U_1 = \frac1n\big\{\E\,U_1^2 - (\E\,U_1)^2\big\} \leq \frac1n\E\,S_1^2 = \O(n^{-1}\an^{-1}\bn),\ee
$\Var\,n^{2/5}R_n \conv 0$ for $\an=n^{-1/5}$ and $\bn = n^{-\beta}$ with $\beta>0$.

Now we consider the expectation of $U_i$.
\be \E\,U_i &=& \int_v\int_u 1_{(z_0,z_0+\bn]}(v)\int_{v-\bn}^{z_0}\tilde k_{\an,\bn}(t_0-u,z-v)\,dz h_1(u,v)\,du\,dv\\
&& \ \ - \int_v\int_u 1_{(z_0-\bn,z_0]}(v)\int_{z_0}^{v+\bn}\tilde k_{\an,\bn}(t_0-u,z-v)\,dz h_1(u,v)\,du\,dv\\
&=& \bn\int_{y=-1}^0\int_{x=-1}^1\int_{w=-1}^{y} \tilde k\left(x,w\right)\,dw h_1(t_0-\an x,z_0-\bn y)\,dx\,dy\\
&& \ \ - \bn\int_{y=0}^1\int_{x=-1}^1\int_{w=y}^1 \tilde k\left(x,w\right)\,dw h_1(t_0-\an x,z_0-\bn y)\,dx\,dy\\
&=& \bn h_1(t_0,z_0)\left\{\int_{x=-1}^1\int_{y=-1}^0\int_{w=-1}^{y} \tilde k(x,w)\,dw\,dy\,dx - \int_{x=-1}^1\int_{y=0}^1\int_{w=y}^1 \tilde k(x,w)\,dw\,dy\,dx\right\}\\
&& \ \ - \bn\an\partial_1h_1(t_0,z_0)\left\{\int_{x=-1}^1\int_{y=-1}^0\int_{w=-1}^{y} x\tilde k(x,w)\,dw\,dy\,dx - \int_{x=-1}^1\int_{y=0}^1\int_{w=y}^1 x\tilde k(x,w)\,dw\,dy\,dx\right\}\\
&& \ \ - \bn^2\partial_2h_1(t_0,z_0)\left\{\int_{x=-1}^1\int_{y=-1}^0\int_{w=-1}^{y} y\tilde k(x,w)\,dw\,dy\,dx - \int_{x=-1}^1\int_{y=0}^1\int_{w=y}^1 y\tilde k(x,w)\,dw\,dy\,dx\right\}\\
&& \ \ + \O\big(\bn\an^2\big) + \O\big(\bn^2\an\big) + \O\big(\bn^3\big)\\
&=& -\bn h_1(t_0,z_0)\int_{x=-1}^1\int_{w=-1}^1 w\tilde k(x,w)\,dw\,dx + \bn\an\partial_1h_1(t_0,z_0)\int_{x=-1}^1\int_{w=-1}^1 xw\tilde k(x,w)\,dw\,dx \\
&& \ \ + \frac12\bn^2\partial_2h_1(t_0,z_0)\int_{x=-1}^1\int_{w=-1}^1 w^2\tilde k(x,w)\,dw\,dx + \O\big(\bn\an^2\big) + \O\big(\bn^2\an\big) + \O\big(\bn^3\big)
\ee
where the last equality follows from changing the order of integration. By condition $(K.3)$, the first two integrals are zero and the last integral equals $m_2(\tilde k)$, so that
\be \E\,n^{2/5}R_n \conv \left\{\begin{array}{ll}
\pm\infty & \text{for } \beta < 1/5,\\
\frac12c_2^2m_2(\tilde k)g(t_0)\partial_2^2F_0(t_0,z_0) & \text{for } \beta = 1/5,\\
0 & \text{for } \beta > 1/5.
\end{array}\right.\ee
Applying Slutsky's Lemma to 
\be n^{2/5}\big(\Fpi{2}(t_0,z_0)-\Fpi{1}(t_0,z_0)\big) = \frac{n^{2/5}R_n}{\gh{n}(t_0)}\ee
gives the result.
\end{proof2}

\newpage
\begin{figure}[p]
\centering
\subfigure{
 \includegraphics[width=0.45\textwidth,height=0.45\textwidth]{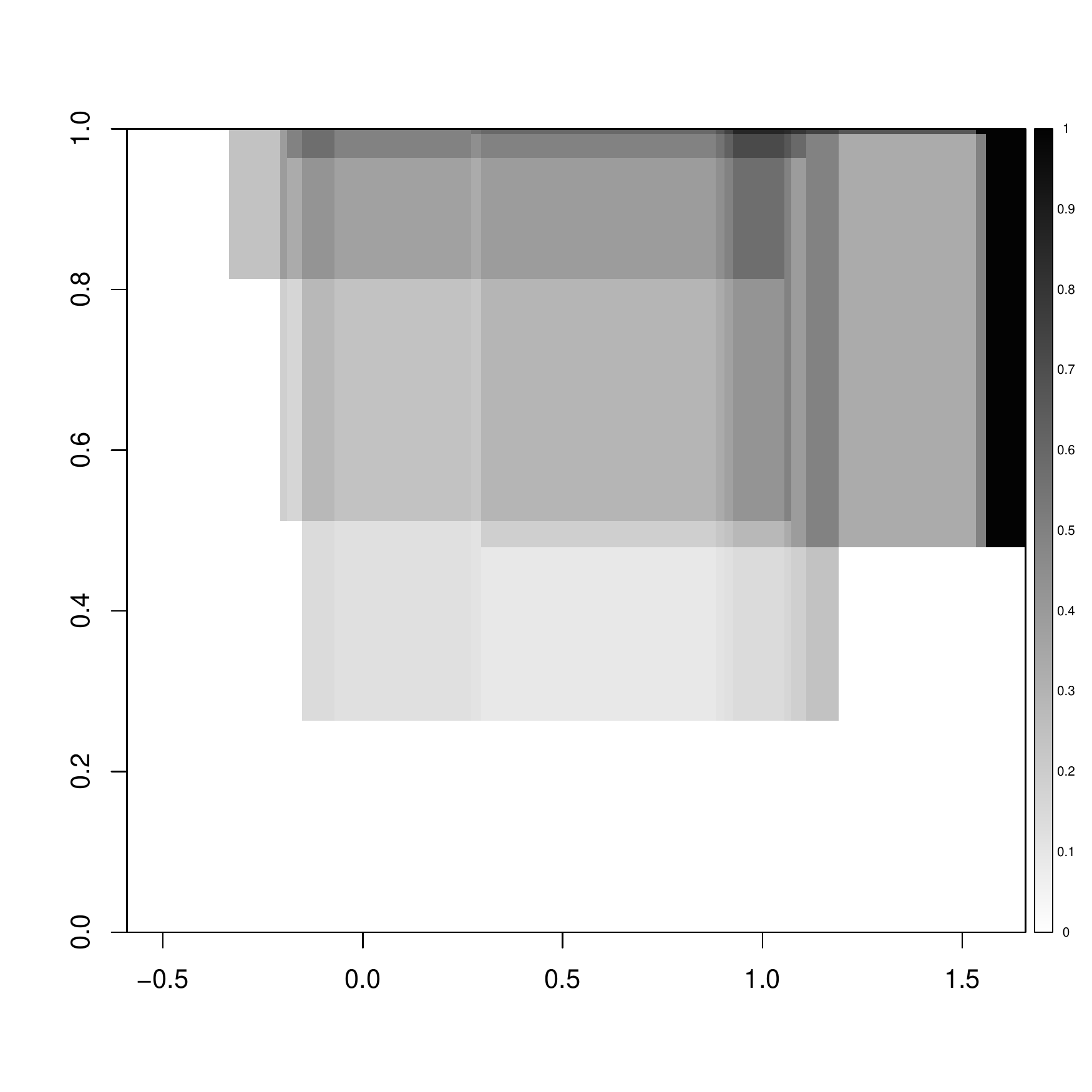} }
\subfigure{
 \includegraphics[width=0.45\textwidth,height=0.45\textwidth]{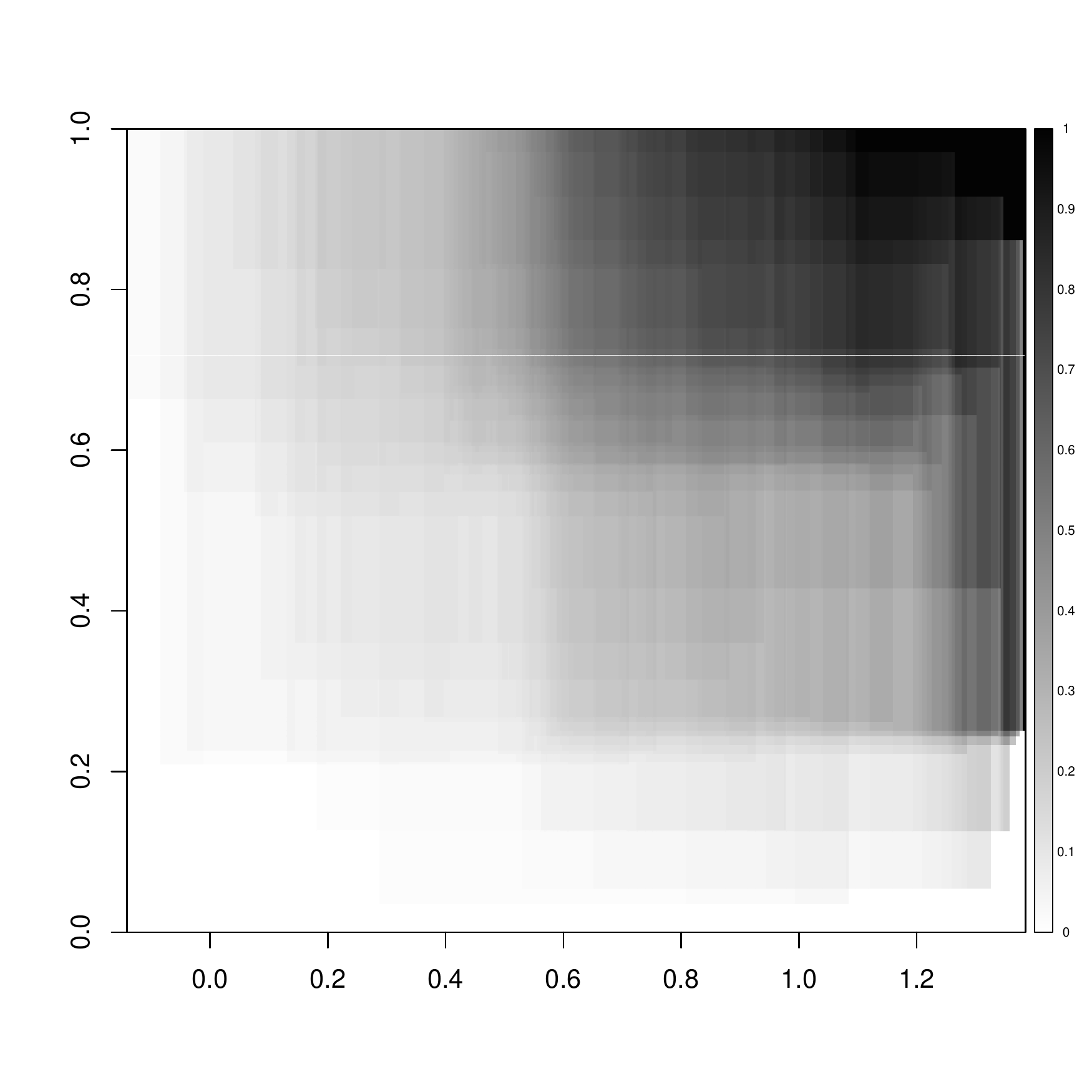} }
\caption{Two examples of the estimator \Fpi{1}\ for a sample of size $n=10$ (left panel) and of size $n=100$ (right panel), $k(x)= \frac121_{[-1,1]}(x)$, $\an=0.65$ (left panel) and $\an=0.40$ (right panel), $F_0(x,y)=xy$ on $[0,1]^2$ and $g(t)=1_{[0,1]}(t)$.}\label{fig:ill_Fn1}
\end{figure}

\begin{figure}[p]
\centering
\subfigure{
 \includegraphics[width=0.45\textwidth,height=0.45\textwidth]{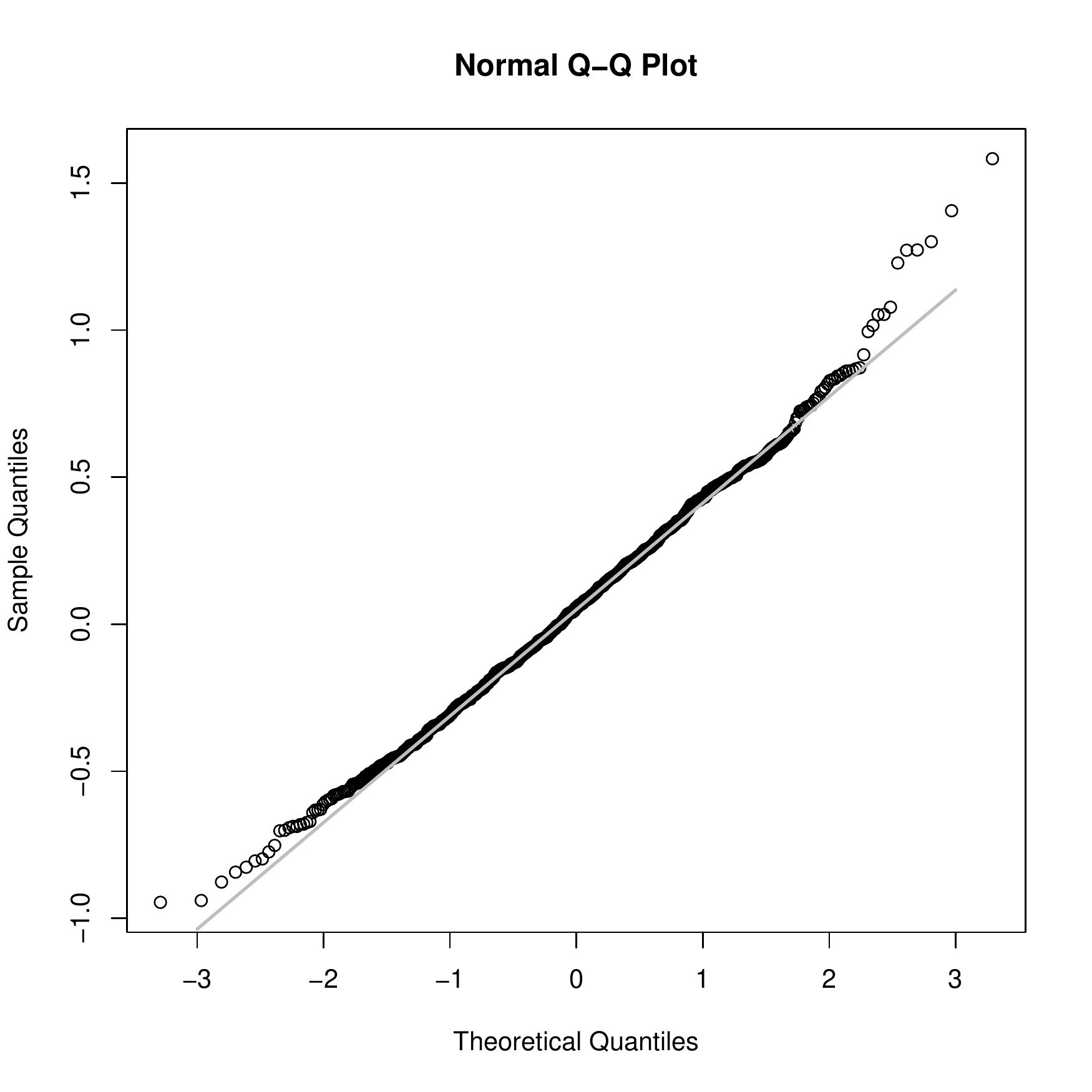} }
\subfigure{
 \includegraphics[width=0.45\textwidth,height=0.45\textwidth]{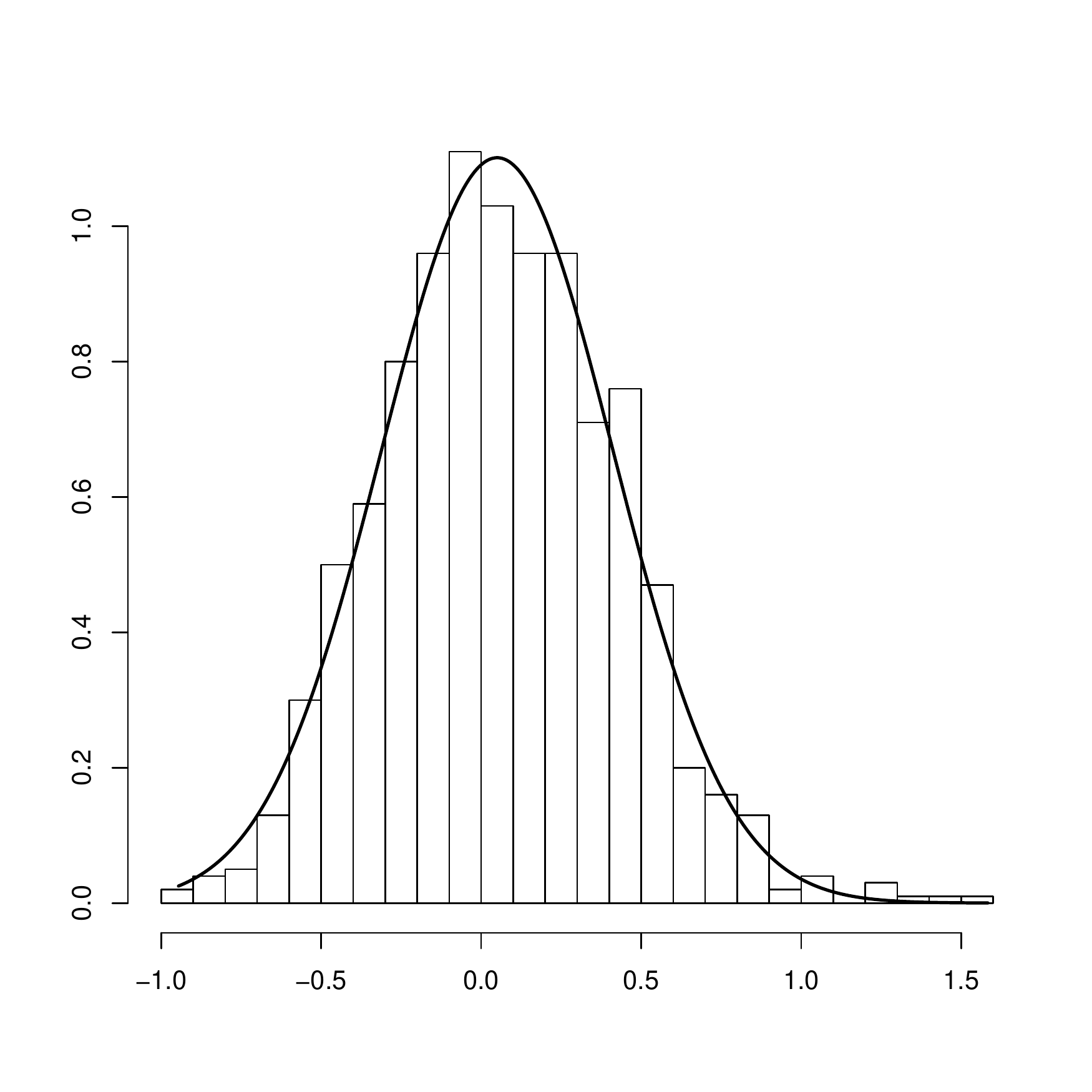} }
\caption{QQ-plot (left panel) and histogram (right panel) of $m=1\,000$ values $n^{2/5}\big(\Fpi{1}(0.5,0.5)-F_0(0.5,0.5)\big)$ for $n=5\,000$, $k(y) = \frac34(1-y^2)1_{[-1,1]}(y)$, $\an=0.09$, $F_0(x,y)=\frac12xy(x+y)$ and $g(t) = 2t$, illustrating Theorem \ref{th:as_Fn1}.}\label{fig:asymp_Fn1}
\end{figure}

\begin{figure}[p]
\centering
\includegraphics[width=0.45\textwidth,height=0.45\textwidth]{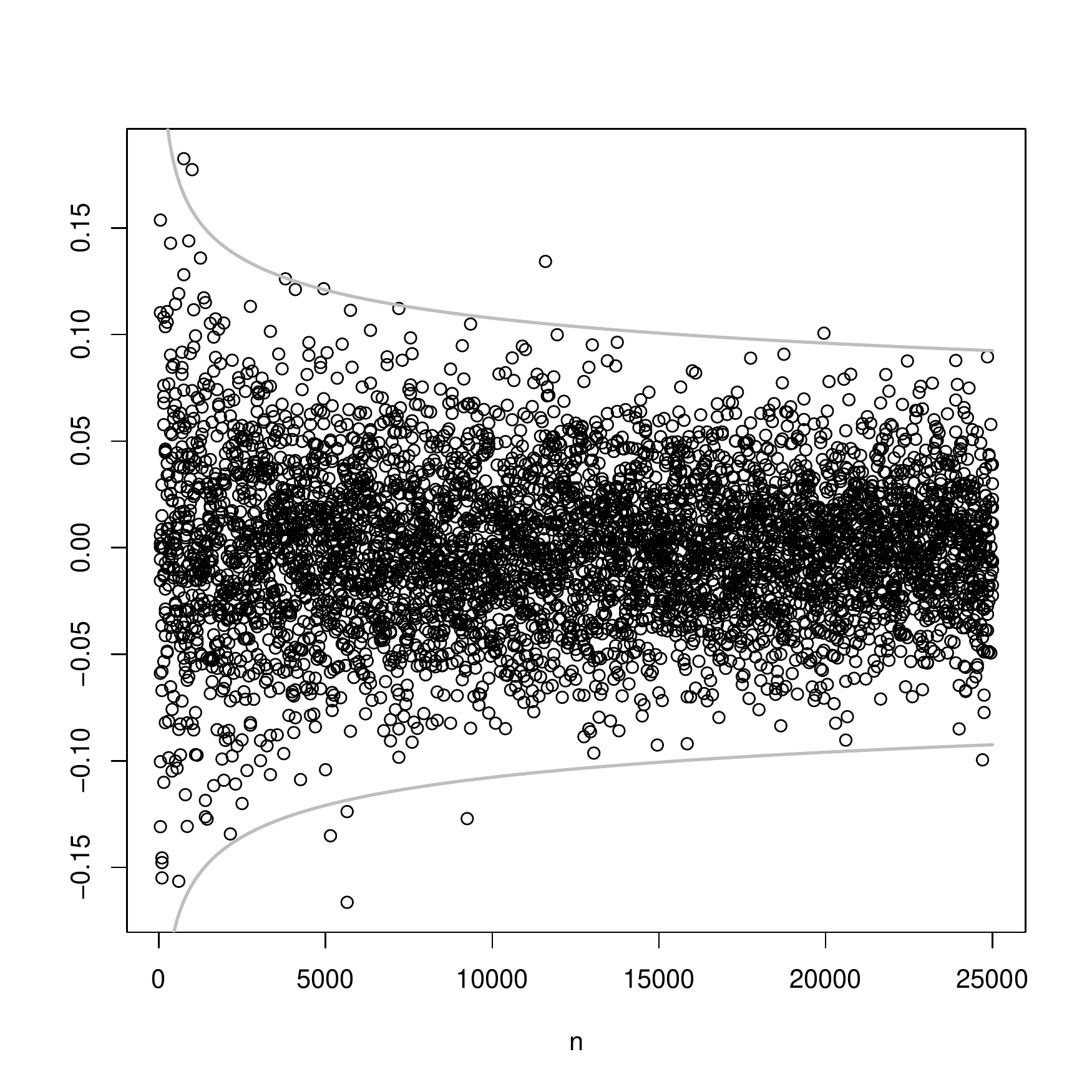}
\caption{Values of $n^{2/5}\big(\Fpi{2}(0.5,0.5)-\Fpi{1}(0.5,0.5)\big)$ as a function of $n$ for $\tilde k(x,y) = k(x)k(y)$, $k(u) = \frac34(1-u^2)$, $\an=\frac12n^{-1/5}$, $\bn = \frac12n^{-1/3}$, $F_0(x,y)=\frac12xy(x+y)$ and $g(t) = 2t$.}\label{fig:Rn}
\end{figure}

\begin{figure}[p]
\centering
\subfigure{
 \includegraphics[width=0.45\textwidth,height=0.45\textwidth]{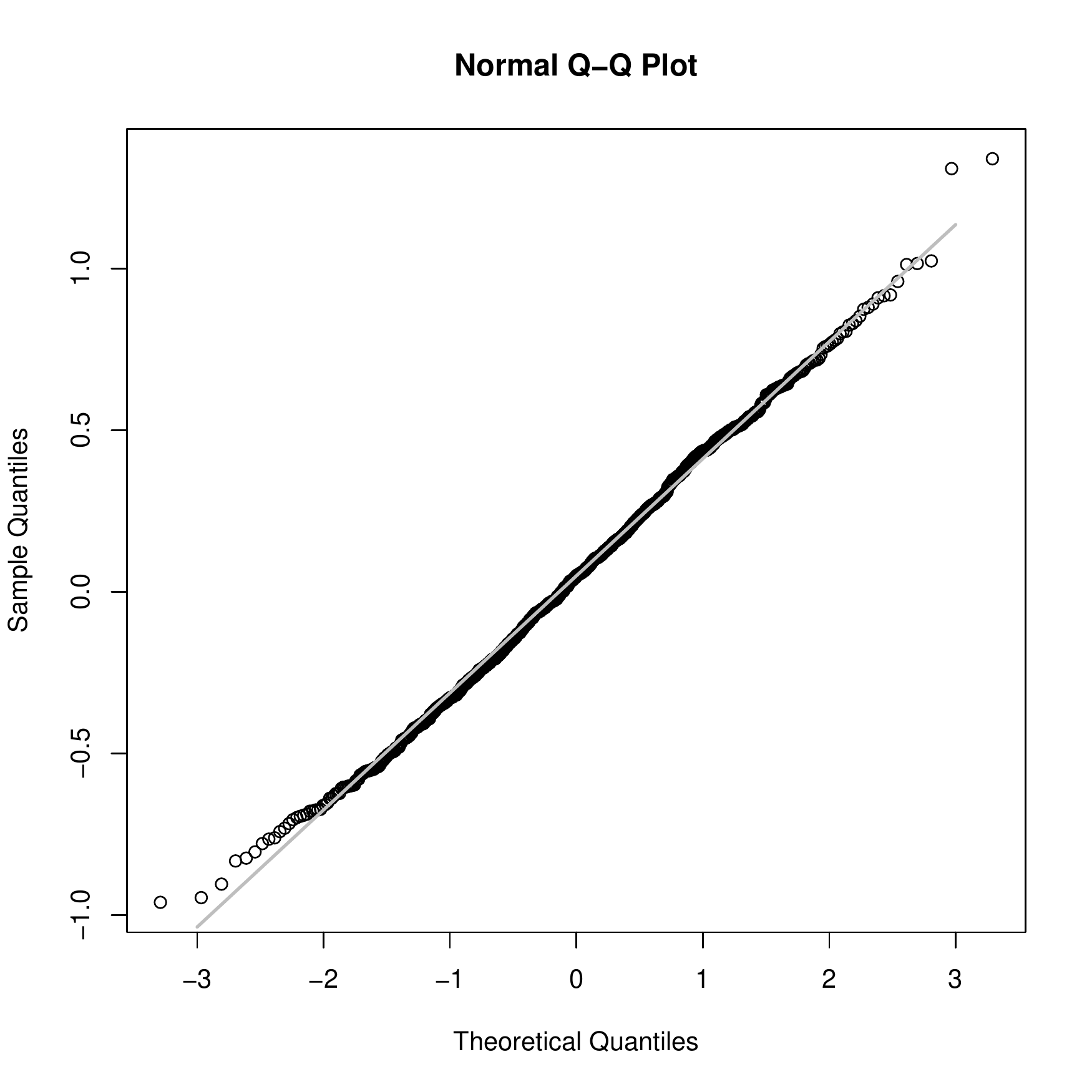} }
\subfigure{
 \includegraphics[width=0.45\textwidth,height=0.45\textwidth]{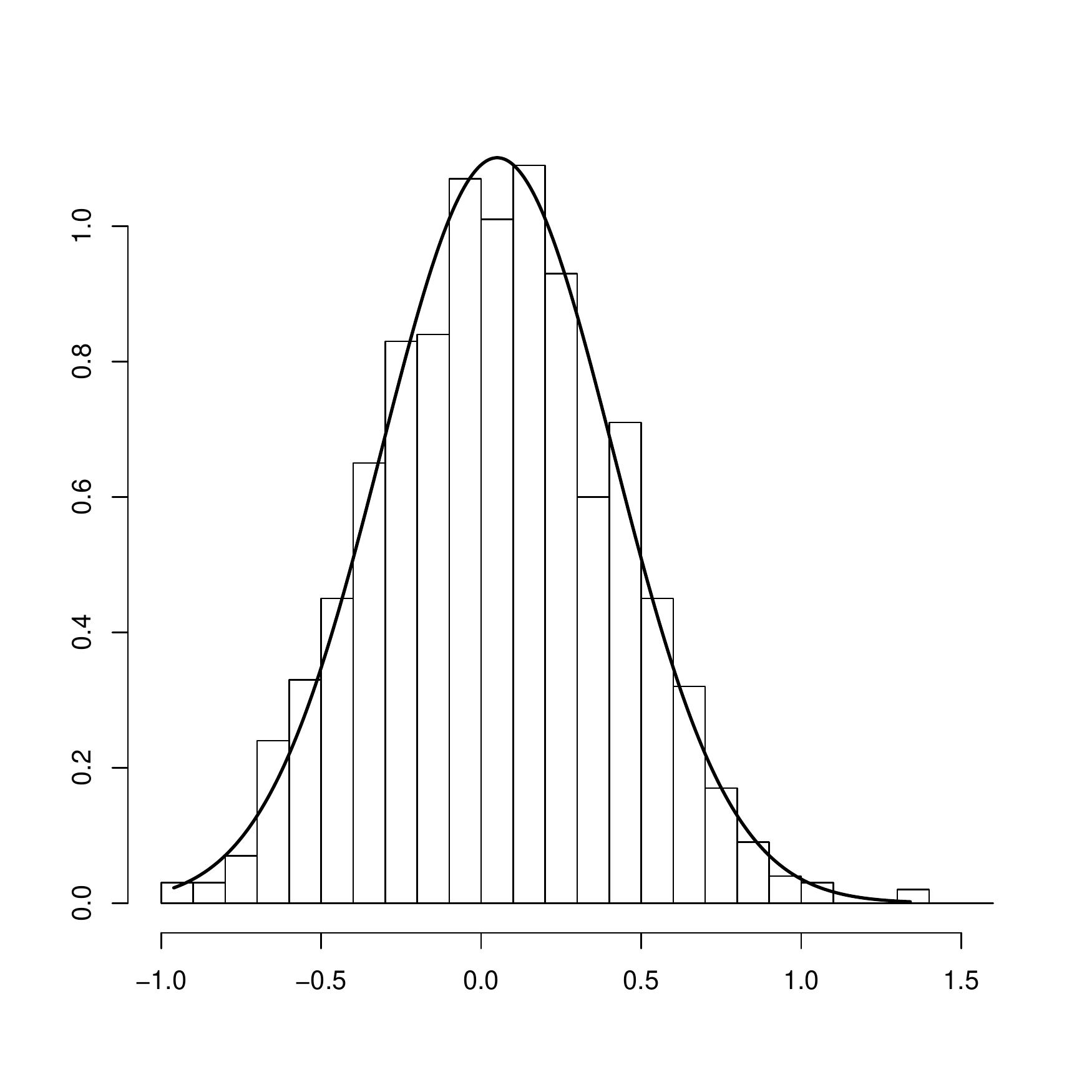} }
\caption{QQ-plot (left panel) and histogram (right panel) of $m=1\,000$ values $n^{2/5}\big(\Fpi{2}(0.5,0.5)-F_0(0.5,0.5)\big)$ for $n=5\,000$, $\tilde k(x,y)=k(x)k(y)$, $k(u) = \frac34(1-u^2)$, \an=0.091, \bn=0.029, $F_0(x,y)=\frac12xy(x+y)$ and $g(t) = 2t$, illustrating Corollary \ref{cor:as_distr_Fn2}.}\label{fig:asymp_Fn2}
\end{figure}

\begin{table}[p]
\centering
\begin{tabular*}{\textwidth}{@{\extracolsep{\fill}}rrcccc}
\hline
\rule[-0.65em]{0cm}{2em}
$(t_0,z_0)$ & $n$ & \an & $\widehat {MSE}$ (s.e.)\phantom{\mbox{\ \ \ }} & (\an,\bn) & $\widehat{MSE}$ (s.e.)\phantom{\mbox{\ \ \ }}\\
\hline\hline
\\
 & & \multicolumn{2}{c}{$\Fpi{1}$} & \multicolumn{2}{c}{$\Fpi{2}$} \\
\hline
(0.4,0.4) & 500 & 0.20 & 5.14$\cdot10^{-4}$ (5.10$\cdot10^{-5}$) & (0.20,0.25) & 4.43$\cdot10^{-4}$ (4.33$\cdot10^{-5}$)\\
 & 1\,000 & 0.20 & 3.31$\cdot10^{-4}$ (3.04$\cdot10^{-5}$) & (0.20,0.15) & 3.10$\cdot10^{-4}$ (3.06$\cdot10^{-5}$)\\
 & 5\,000 & 0.15 & 8.09$\cdot10^{-5}$ (8.47$\cdot10^{-6}$) & (0.15,0.10) & 7.74$\cdot10^{-5}$ (8.28$\cdot10^{-6}$)\\
 & 10\,000 & 0.15 & 4.50$\cdot10^{-5}$ (3.43$\cdot10^{-6}$) & (0.15,0.05) & 4.50$\cdot10^{-5}$ (3.38$\cdot10^{-6}$)\\
\\
(0.6,0.6) & 500 & 0.25 & 8.21$\cdot10^{-4}$ (7.48$\cdot10^{-5}$) & (0.25,0.15) & 7.82$\cdot10^{-4}$ (7.04$\cdot10^{-5}$)\\
 & 1\,000 & 0.20 & 5.31$\cdot10^{-4}$ (4.34$\cdot10^{-5}$) & (0.20,0.05) & 5.31$\cdot10^{-4}$ (4.33$\cdot10^{-5}$)\\
 & 5\,000 & 0.15 & 1.21$\cdot10^{-4}$ (9.98$\cdot10^{-6}$) & (0.15,0.05) & 1.21$\cdot10^{-4}$ (9.88$\cdot10^{-6}$)\\
 & 10\,000 & 0.15 & 9.21$\cdot10^{-5}$ (7.41$\cdot10^{-6}$) & (0.15,0.05) & 9.14$\cdot10^{-5}$ (7.31$\cdot10^{-6}$)\\
\\
 & & \multicolumn{2}{c}{$\tilde F_n$} & \multicolumn{2}{c}{$\FMS$} \\
\hline
(0.4,0.4) & 500 & 0.200 & 5.56$\cdot10^{-4}$ (4.56$\cdot10^{-5}$) & (0.250,0.250) & 7.21$\cdot10^{-4}$ (5.58$\cdot10^{-5}$)\\
 & 1\,000 & 0.100 & 3.26$\cdot10^{-4}$ (2.83$\cdot10^{-5}$) & (0.200,0.500) & 3.48$\cdot10^{-4}$ (3.30$\cdot10^{-5}$)\\
 & 5\,000 & 0.100 & 1.10$\cdot10^{-4}$ (9.98$\cdot10^{-6}$) & (0.200,0.333) & 7.20$\cdot10^{-5}$ (7.11$\cdot10^{-6}$)\\
 & 10\,000 & 0.067 & 6.38$\cdot10^{-5}$ (4.82$\cdot10^{-6}$) & (0.167,0.333) & 7.35$\cdot10^{-5}$ (6.45$\cdot10^{-6}$)\\
\\
(0.6,0.6) & 500 & 0.200 & 1.45$\cdot10^{-3}$ (1.35$\cdot10^{-4}$) & (0.250,0.250) & 5.51$\cdot10^{-4}$ (5.28$\cdot10^{-5}$)\\
 & 1\,00 & 0.250 & 3.59$\cdot10^{-3}$ (1.97$\cdot10^{-4}$) & (0.250,0.200) & 4.13$\cdot10^{-4}$ (3.40$\cdot10^{-5}$)\\
 & 5\,000 & 0.333 & 1.54$\cdot10^{-2}$ (2.03$\cdot10^{-4}$) & (0.250,0.167) & 2.24$\cdot10^{-4}$ (5.66$\cdot10^{-5}$)\\
 & 10\,000 & 0.333 & 1.50$\cdot10^{-2}$ (1.48$\cdot10^{-4}$) & (0.250,0.200) & 1.32$\cdot10^{-4}$ (7.23$\cdot10^{-6}$)\\

\end{tabular*}
\caption{\it Minimum values of the estimated MSE of the estimators \Fpi{1}, \Fpi{2}, $\tilde F_n$ and \FMS\ for different values of $n$ at different points $(t_0,z_0)$ for the simulation study. The values of \an\ and \bn\ that resulted in these minimal values are also given, as well as the standard errors of the mean of the squared differences between the estimator and the true value.}\label{table:sim_study}
\end{table}

\begin{figure}[p]
\centering
\subfigure[$n=500$ with $\bn = 0.15$ for \Fpi{2}]{
 \includegraphics[width=0.48\textwidth]{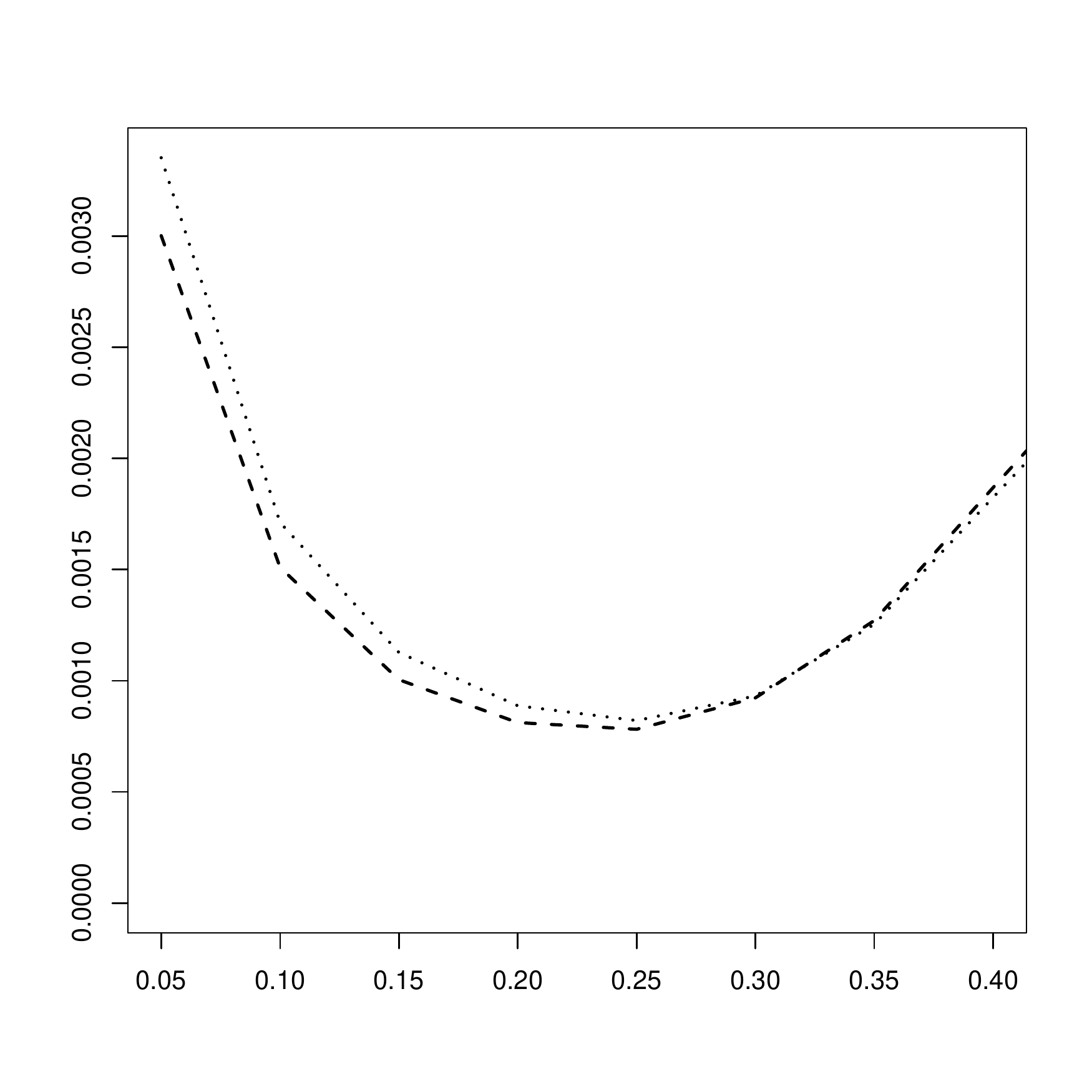}}
\subfigure[ $n=1\,000$ with $\bn = 0.05$ for \Fpi{2}]{
 \includegraphics[width=0.48\textwidth]{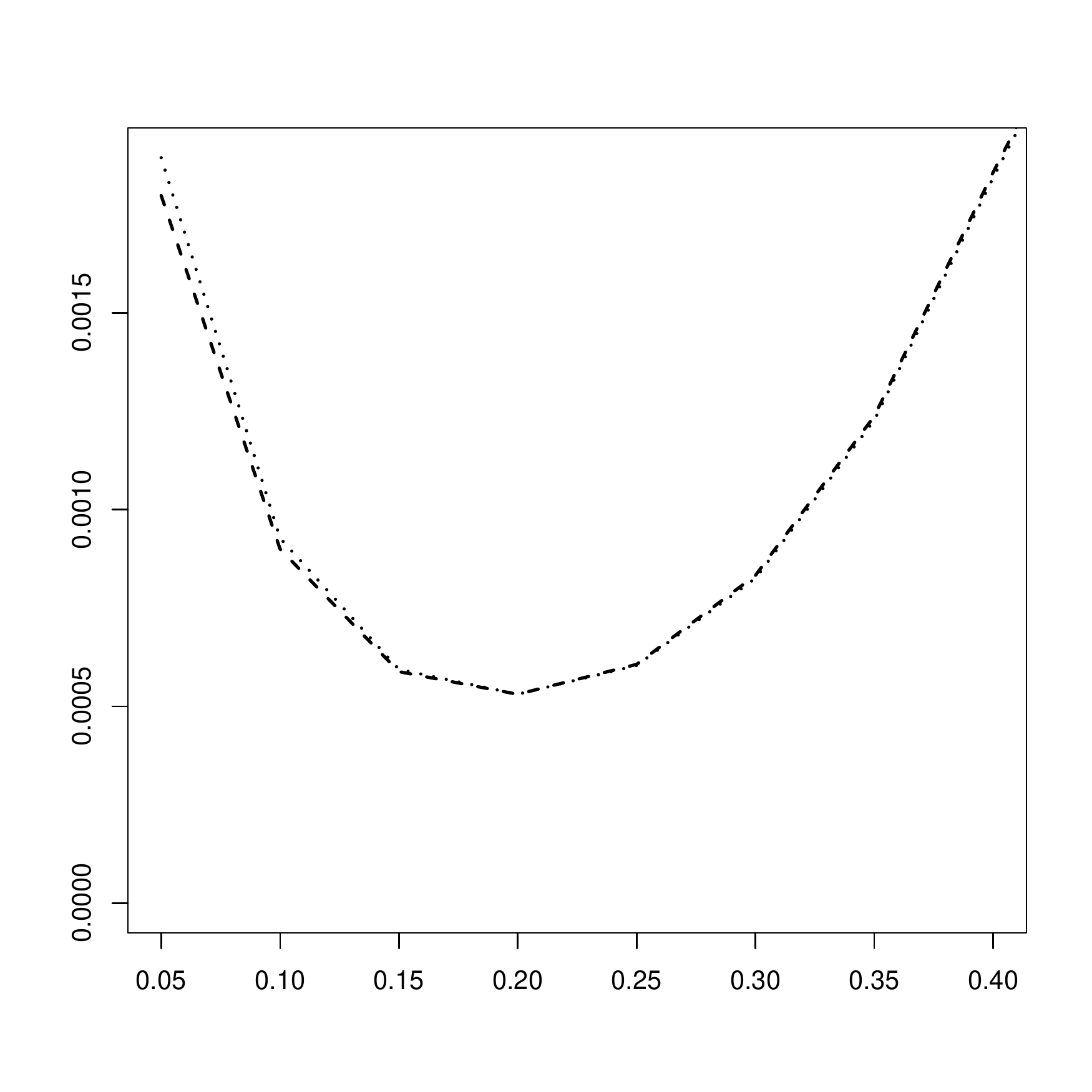}}
\subfigure[$n=5\,000$ with $\bn = 0.05$ for \Fpi{2}]{
 \includegraphics[width=0.48\textwidth]{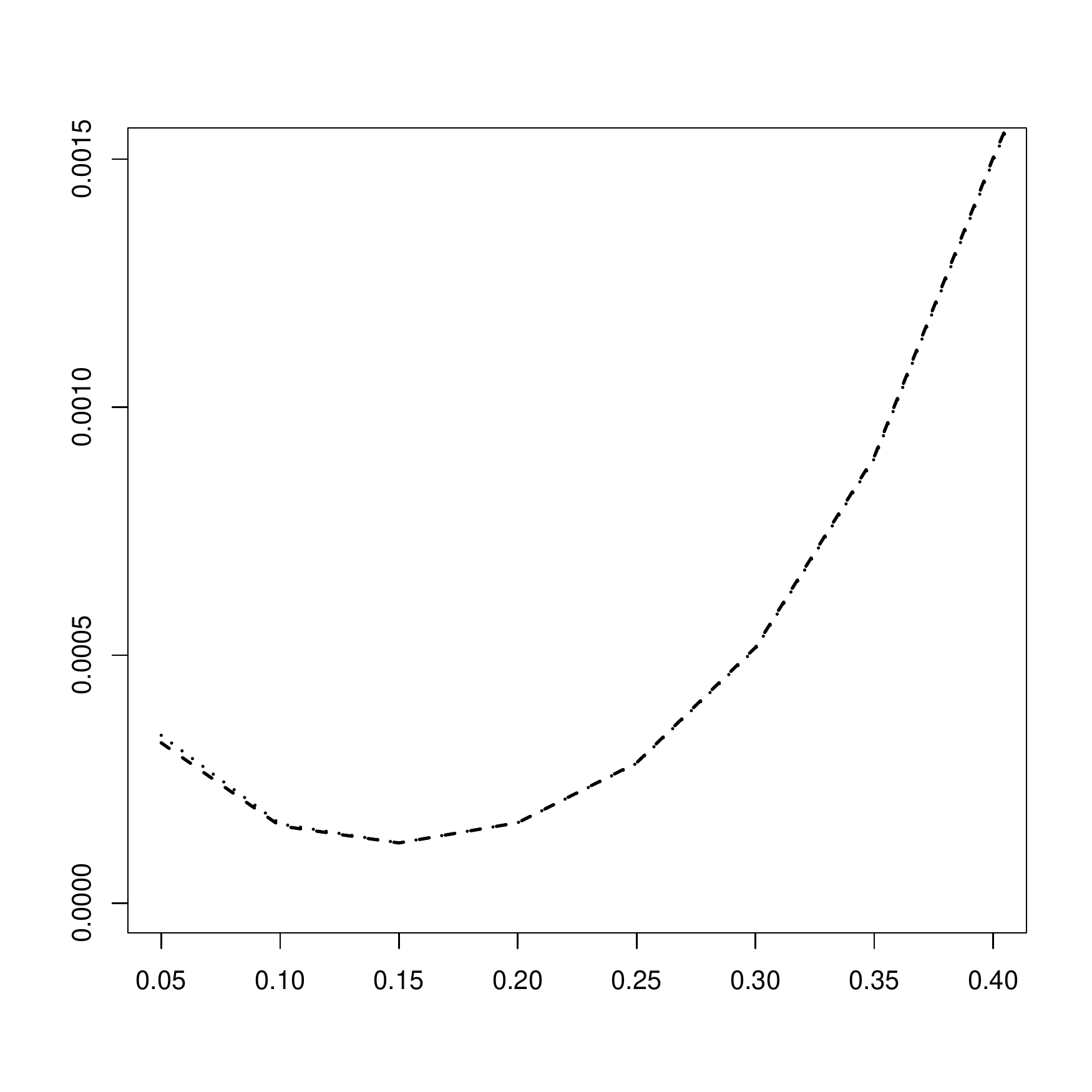}}
\subfigure[$n=10\,000$ with $\bn = 0.05$ for \Fpi{2}]{
 \includegraphics[width=0.48\textwidth]{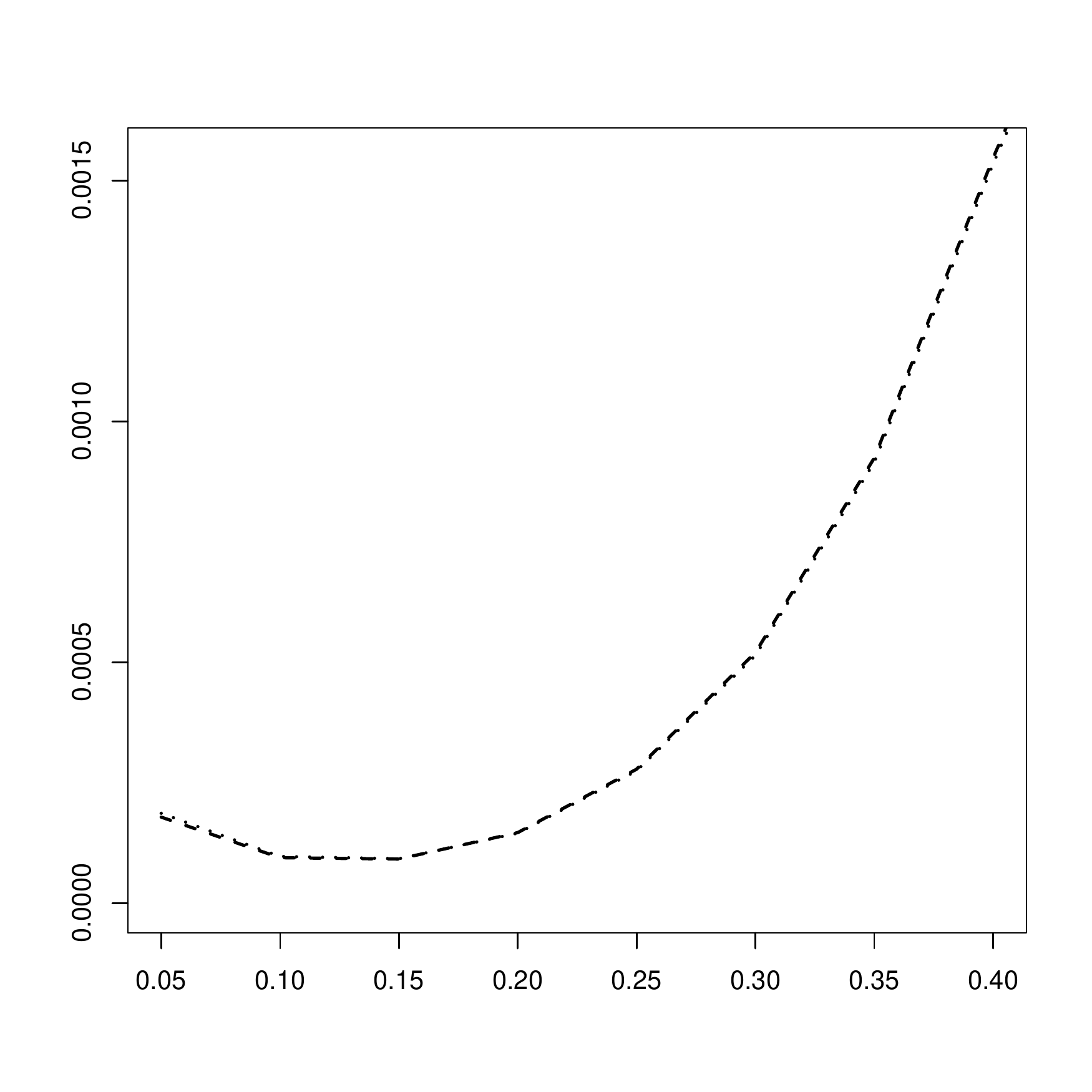}}
\caption{\it The estimated MSE as function of the smoothing parameter $\an$ the estimators \Fpi{1}\ (dotted line) and \Fpi{2}\ (dashed line) for different values of $n$ at the point $(0.6,0.6)$ for the simulation study.}\label{fig:sim_study}
\end{figure}

\begin{figure}[p]
\centering
\includegraphics[width=0.45\textwidth,height=0.45\textwidth]{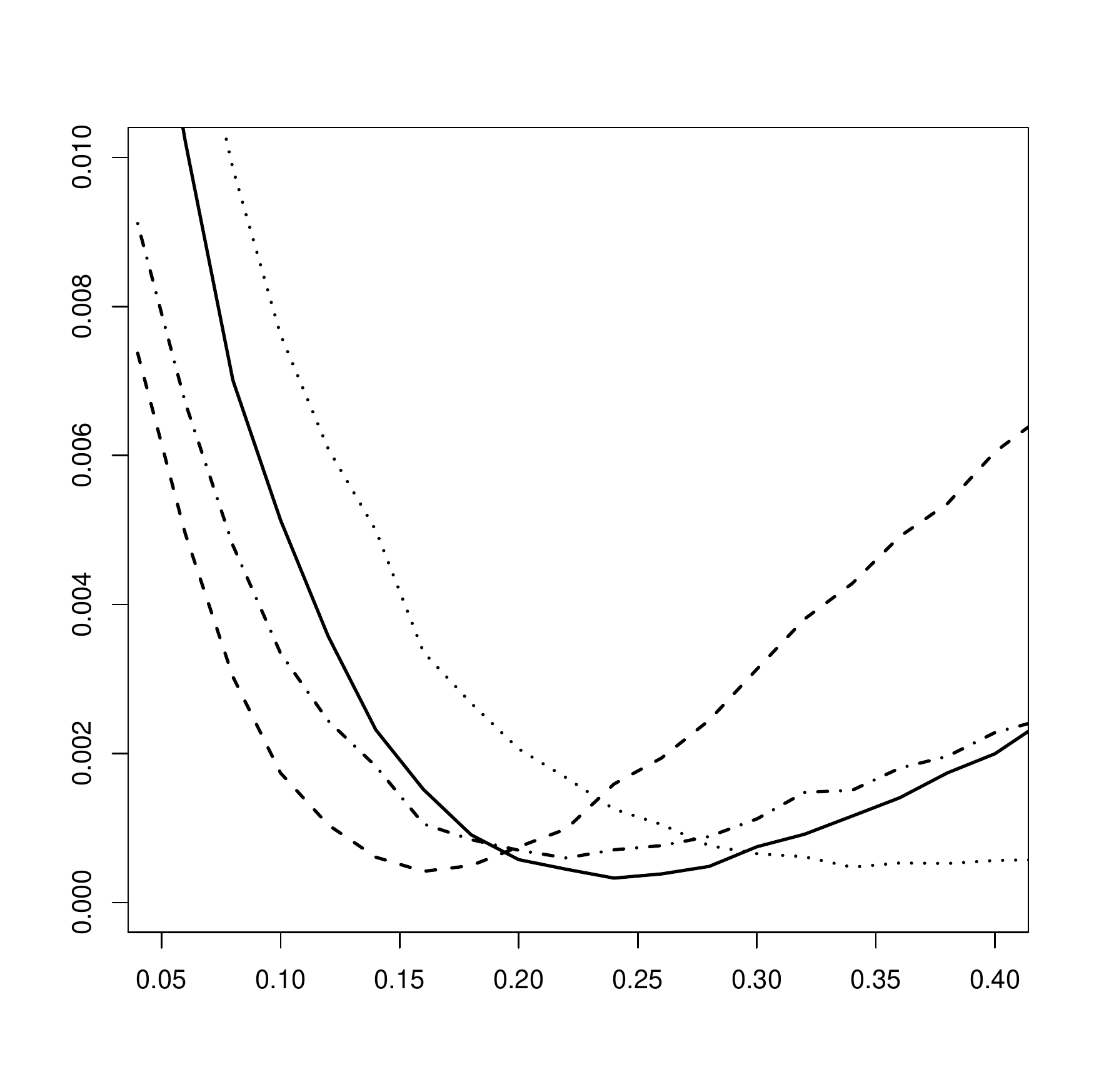}
\caption{Values the estimated MSEs $\widehat{MSE}^{(1)}(\an;0.5,0.5)$ (dotted line), $\widehat{MSE}^{(2)}(\an,0.78;0.5,0.5)$ (solid line), $\widetilde{MSE}^{(1)}(\an;0.5,0.5)$ (dash-dotted line) and $\widetilde{MSE}^{(2)}(\an,0.8;0.5,0.5)$ (dashed line) as function of \an.}\label{fig:bw_sel}
\end{figure}


\begin{thebibliography}{18}
\providecommand{\natexlab}[1]{#1}
\providecommand{\url}[1]{\texttt{#1}}
\providecommand{\urlprefix}{URL }

\bibitem[\protect\citeauthoryear{Burke}{Burke}{1988}]{burke:88}
Burke, M.~D. (1988), Estimation of a bivariate distribution function under random censorship, \textit{Biometrika}, \textbf{75}: 379--382.

\bibitem[\protect\citeauthoryear{Cacoullos}{Cacoullos}{1964}]{cacoullos:64}
Cacoullos, T. (1964), Estimation of a multivariate density, \textit{Ann. Inst. Statist. Math.}, \textbf{18}: 179--189.

\bibitem[\protect\citeauthoryear{Flynn, Forthal, and {The rgp120 HIV Vaccine Study Group}}{Flynn et~al.}{2005}]{rgp120:05}
Flynn, N.~M., Forthal, D.~N., and {The rgp120 HIV Vaccine Study Group} (2005), Placebo-controlled phase 3 trial of a recombinant glycoprotein 120 vaccine to prevent {HIV}-1 infection, \textit{Journal of Infectious Diseases}, \textbf{191}: 654--665.

\bibitem[\protect\citeauthoryear{Gilbert, Self, Rao, Naficy, and Clemens}{Gilbert et~al.}{2001}]{gilbert:01}
Gilbert, P., Self, S., Rao, M., Naficy, A., and Clemens, J. (2001), Sieve analysis: methods for assessing from vaccine trial data how vaccine efficacy  varies with genotypic and phenotypic pathogen variation, \textit{Journal of Clinical Epidemiology}, \textbf{54}: 68--85.

\bibitem[\protect\citeauthoryear{Groeneboom}{Groeneboom}{1996}]{groeneboom:96}
Groeneboom, P. (1996), Lectures on inverse problems, in: \textit{Lectures on Probability Theory and Statistics}, \textit{Lecture Notes in Mathematics}, volume 1648, 67--164, Springer, Berlin.

\bibitem[\protect\citeauthoryear{Groeneboom, Jongbloed, and Witte}{Groeneboom et~al.}{2010}]{witte:10}
Groeneboom, P., Jongbloed, G., and Witte, B.~I. (2010), Maximum smoothed likelihood estimation and smoothed maximum likelihood estimation in the current status model, \textit{Ann. Statist.}, \textbf{38}: 352--387.

\bibitem[\protect\citeauthoryear{Groeneboom, Maathuis, and Wellner}{Groeneboom et~al.}{2008{\natexlab{a}}}]{groeneboom:08}
Groeneboom, P., Maathuis, M.~H., and Wellner, J.~A. (2008{\natexlab{a}}), Current status data with competing risks: consistency and rates of convergence of the {MLE}, \textit{Ann. Statist.}, \textbf{36}: 1031--1063.

\bibitem[\protect\citeauthoryear{Groeneboom, Maathuis, and Wellner}{Groeneboom et~al.}{2008{\natexlab{b}}}]{groeneboom:08b}
Groeneboom, P., Maathuis, M.~H., and Wellner, J.~A. (2008{\natexlab{b}}), Current status data with competing risks: limiting distribtuion of the {MLE}, \textit{Ann. Statist.}, \textbf{36}: 1064--1089.

\bibitem[\protect\citeauthoryear{Hall and Smith}{Hall and Smith}{1988}]{hall:88}
Hall, P. and Smith, R.~L. (1988), The kernel method for unfolding sphere size distributions, \textit{J. Comput. Phys.}, \textbf{74}: 409 -- 421.

\bibitem[\protect\citeauthoryear{H{\"a}rdle, Janssen, and Serfling}{H{\"a}rdle et~al.}{1988}]{haerdle:88}
H{\"a}rdle, W., Janssen, P., and Serfling, R. (1988), Strong uniform consistency rates for estimators of conditional functionals, \textit{Ann. Statist.}, \textbf{88}: 1428--1449.

\bibitem[\protect\citeauthoryear{Hudgens, Maathuis, and Gilbert}{Hudgens et~al.}{2007}]{hudgens:07}
Hudgens, M.~G., Maathuis, M.~H., and Gilbert, P.~B. (2007), Nonparametric estimation of the joint distribution of a survival time subject to interval censoring and a continuous mark variable, \textit{Biometrics}, \textbf{63}: 372--380.

\bibitem[\protect\citeauthoryear{Maathuis and Wellner}{Maathuis and Wellner}{2008}]{maathuis:08}
Maathuis, M.~H. and Wellner, J.~A. (2008), Inconsistency of the {MLE} for the joint distribution of interval censored survival times and continuous marks, \textit{Scand. J. Statist.}, \textbf{35}: 83--103.

\bibitem[\protect\citeauthoryear{Marron and Padgett}{Marron and Padgett}{1987}]{marron:87}
Marron, J.~S. and Padgett, W.~J. (1987), Asymptotically optimal bandwidth selection for kernel density estimators from randomly right-censored samples, \textit{Ann. Statist.}, \textbf{15}: 1520--1535.

\bibitem[\protect\citeauthoryear{Mokkadem, Pelletier, and Worms}{Mokkadem et~al.}{2005}]{mokkadem:05}
Mokkadem, A., Pelletier, M., and Worms, J. (2005), Large and moderate deviations principles for kernel estimation of a multivariate density and its partial derivatives, \textit{Aust. N. Z. J. Stat.}, \textbf{47}: 489--502.

\bibitem[\protect\citeauthoryear{Patil, Wells, and Marron}{Patil et~al.}{1994}]{patil:94}
Patil, P.~N., Wells, M.~T., and Marron, J.~S. (1994), Some heuristics of kernel based estimators of ratio functions, \textit{J. Nonparametr. Stat.}, \textbf{4}: 203--209.

\bibitem[\protect\citeauthoryear{Rudemo}{Rudemo}{1982}]{rudemo:82}
Rudemo, M. (1982), Empirical choice of histograms and kernel density estimators, \textit{Scandinavian Journal of Statistics}, \textbf{9}: 65--78.

\bibitem[\protect\citeauthoryear{Silverman}{Silverman}{1978}]{silverman:78}
Silverman, B.~W. (1978), Weak and strong uniform consistency of the kernel estimate of a density and its derivative, \textit{Ann. Statist.}, \textbf{6}: 177--184.

\bibitem[\protect\citeauthoryear{Stefanski and Carroll}{Stefanski and Carroll}{1990}]{stefanski:90}
Stefanski, L. and Carroll, R.~J. (1990), Deconvoluting kernel density estimators, \textit{Statistics}, \textbf{21}: 169--184.

\end{thebibliography}
\end{document}